\documentclass[reqno]{amsart}
\pdfoutput=1

\usepackage{amssymb}
\usepackage{color}
\usepackage{enumitem}
\usepackage[T1]{fontenc}
\usepackage[utf8]{inputenc}
\usepackage[l2tabu,orthodox]{nag}
\usepackage{stmaryrd}
\usepackage{tensor}
\usepackage[vcentermath]{youngtab}

\usepackage[
		pdftex,
		bookmarks,
		bookmarksopen,
		bookmarksopenlevel=1,
		bookmarksnumbered,
		breaklinks,
	]{hyperref}

\hypersetup{
	pdftitle    = Algebraic Integrability Conditions for Killing Tensors on Constant Sectional Curvature Manifolds,%
	pdfauthor   = Konrad P. Sch\"obel,%
	pdfsubject  = Riemannian geometry,%
	pdfcreator  = \LaTeX{} with `scrreprt' class,%
	pdfkeywords = Killing tensors - integrability - constant sectional curvature,%
	pdfproducer = pdf\LaTeX%
}
\definecolor{linkcolor}{rgb}{0.375,0.0,0.0}
\definecolor{citecolor}{rgb}{0.0,0.375,0.0}
\definecolor{urlcolor} {rgb}{0.0,0.0,0.375}

\title
	[Integrability of \name{Killing} Tensors on Constant Curvature Manifolds]
	{Algebraic Integrability Conditions \\ for \name{Killing} Tensors \\ on Constant Sectional Curvature Manifolds}
\author{Konrad P. Schöbel}
\email{konrad.schoebel@uni-jena.de}
\address{
	Mathematisches Institut \\
	Fakultät für Mathematik und Informatik \\
	Friedrich-Schiller-Universität Jena \\
	07737 Jena \\
	Germany
}
\thanks{\textit{E-mail.} \href{mailto:konrad.schoebel@uni-jena.de}{\nolinkurl{konrad.schoebel@uni-jena.de}}}

\newtheorem{theorem}{Theorem}[section]
\newtheorem*{maintheorem}{Main Theorem}
\newtheorem*{maincorollary}{Main Corollary}
\newtheorem{proposition}[theorem]{Proposition}
\newtheorem{lemma}[theorem]{Lemma}
\newtheorem{corollary}[theorem]{Corollary}
\newtheorem{definition}[theorem]{Definition}
\newtheorem{remark}[theorem]{Remark}

\newtheorem{example}[theorem]{Example}
\newtheorem{examples}[theorem]{Examples}
\newtheorem{convention}[theorem]{Convention}
\numberwithin{equation}{section}

\setenumerate[1]{label=(\roman*),ref=(\roman*)}
\setenumerate[2]{label=(\alph*),ref=(\alph*)}

\newcommand{\linkref}[2]{#1 \ref{#2}}
\newcommand{\enumskip}{\mbox{}}

\newcommand{\name}[1]{{#1}}
\newcommand{\dfn}[1]{\emph{#1}}

\makeatletter
	\newcommand{\phantomequation}[2]{\immediate\write\@auxout{\string\newlabel{#1}{{\theequation#2}{\thepage}{}{equation.\theequation}{}}}}
	\newcommand{\scalebox}[2]{%
		\leavevmode
		\def\Gscale@x{#1}\def\Gscale@y{#1}%
		\setbox\z@\hbox{{#2}}%
		\setbox\tw@\hbox{\Gscale@start\rlap{\copy\z@}\Gscale@end}%
		\ifdim#1\p@<\z@
			\ht\tw@-#1\dp\z@
			\dp\tw@-#1\ht\z@
		\else
			\ht\tw@#1\ht\z@
			\dp\tw@#1\dp\z@
		\fi
		\ifdim#1\p@<\z@
			\hb@xt@-#1\wd\z@{\kern-#1\wd\z@\box\tw@\hss}%
		\else
			\wd\tw@#1\wd\z@
			\box\tw@
		\fi}
\makeatother

\newcommand{\scaleboxfactor}{0.6}
\newcommand{\smallscaleboxfactor}{0.5}
\newcommand{\textYF}{\Yboxdim4pt}
\newcommand{\YF}{\Yboxdim6pt}

\newcommand{\aone}{\scalebox{\scaleboxfactor}{\ensuremath{a_1}}}
\newcommand{\atwo}{\scalebox{\scaleboxfactor}{\ensuremath{a_2}}}
\newcommand{\bone}{\scalebox{\scaleboxfactor}{\ensuremath{b_1}}}
\newcommand{\btwo}{\scalebox{\scaleboxfactor}{\ensuremath{b_2}}}
\newcommand{\cone}{\scalebox{\scaleboxfactor}{\ensuremath{c_1}}}
\newcommand{\ctwo}{\scalebox{\scaleboxfactor}{\ensuremath{c_2}}}
\newcommand{\done}{\scalebox{\scaleboxfactor}{\ensuremath{d_1}}}
\newcommand{\dtwo}{\scalebox{\scaleboxfactor}{\ensuremath{d_2}}}
\newcommand{\eone}{\scalebox{\scaleboxfactor}{\ensuremath{e_1}}}
\newcommand{\etwo}{\scalebox{\scaleboxfactor}{\ensuremath{e_2}}}

\newcommand{\ftwo}{\scalebox{\scaleboxfactor}{\ensuremath{f_2}}}
\newcommand{\gone}{\scalebox{\scaleboxfactor}{\ensuremath{g_1}}}
\newcommand{\gtwo}{\scalebox{\scaleboxfactor}{\ensuremath{g_2}}}

\newcommand{\htwo}{\scalebox{\scaleboxfactor}{\ensuremath{h_2}}}
\newcommand{\sone}{\scalebox{\scaleboxfactor}{\ensuremath{s_1}}}
\newcommand{\ssp }{\scalebox{\scaleboxfactor}{\ensuremath{s_p}}}
\newcommand{\asq }{\scalebox{\scaleboxfactor}{\ensuremath{a_q}}}
\newcommand{\smallhdots}{\scalebox{\scaleboxfactor}{\ensuremath{\cdots}}}
\newcommand{\smallvdots}{\scalebox{\smallscaleboxfactor}{\ensuremath{\vdots}}}

\renewcommand{\ge}{\geqslant}

\newcommand{\R}{\mathbb{R}}

\newcommand{\comma}{\,,}
\newcommand{\fullstop}{\,.}
\newcommand{\coloneq}{:=}

\newcommand{\isom}{\cong}
\newcommand{\adjoint}{\star}

\DeclareMathOperator{\tr}{tr}
\DeclareMathOperator{\sign}{sign}
\DeclareMathOperator{\im}{im}
\DeclareMathOperator{\Sym}{Sym}
\DeclareMathOperator{\End}{End}
\DeclareMathOperator{\SG}{S}
\DeclareMathOperator{\GLG}{GL}
\DeclareMathOperator{\OG}{O}

\newcommand{\hook}[1]{\ensuremath{h_{#1}}}

\begin{document}

\begin{abstract}

	We use an isomorphism between the space of valence two Killing tensors on an $n$-dimensional constant sectional curvature
	manifold and the irreducible $\GLG(n+1)$-representation space of algebraic curvature tensors
	\cite{McLenaghan&Milson&Smirnov} in order to translate the Nijenhuis integrability conditions for a Killing tensor into
	purely algebraic integrability conditions for the corresponding algebraic curvature tensor, resulting in two simple
	algebraic equations of degree two and three.  As a first application of this we construct a new family of integrable Killing
	tensors.

\end{abstract}

\maketitle

\section{Introduction}

Besides the \name{Euler}-\name{Lagrange} formalism and the \name{Hamilton} formalism, the \name{Hamilton}-\name{Jacobi} equation
is one of the three fundamental reformulations of classical \name{Newton}ian mechanics with wide applications in physics as well
as mathematics, ranging from classical mechanics over optics and semi-classical quantum mechanics to \name{Riemann}ian geometry.
In many cases this first-order non-linear partial differential equation can be solved by a separation of variables after
choosing appropriate coordinates.  It is therefore a classical problem in \name{Riemann}ian geometry to classify such
coordinates \cite{Staeckel,Levi-Civita,Eisenhart,Kalnins&Miller80}.  The \name{Hamilton}-\name{Jacobi} equation separates in a
given system of orthogonal coordinates if and only if there exists an integrable valence two \name{Killing} tensor field with
simple eigenvalues whose eigenvectors are tangent to the coordinate lines and such that the potential satisfies a certain
compatibility condition involving this \name{Killing} tensor \cite{Benenti93}.  Integrable \name{Killing} tensors are thus an
important tool in the study of the seprarability of the \name{Hamilton}-\name{Jacobi} equation.  The present work focusses on
the integrability condition for \name{Killing} tensors.

\name{Killing} tensors form a linear space which is invariant under the pullback action of the manifold's isometry group.  In
other words they constitute a representation space of the isometry group.  \name{McLenaghan}, \name{Milson} and \name{Smirnov}
identified this representation in the case of constant sectional curvature manifolds as a certain irreducible representation of
the general linear group \cite{McLenaghan&Milson&Smirnov}.  More precisely, they used the isometric embeddings of the standard
models of constant sectional curvature manifolds as hypersurfaces $M$ in a \name{Euclid}ean vector space $(V,g)$ in order to
write \name{Killing} tensors as restrictions of homogeneous polynomials on $V$, where the coefficients obey certain symmetry
relations.  This yields in particular an explicit natural isomorphism between the space of valence two \name{Killing} tensors on
$M$ and the irreducible $\GLG(V)$-representation space of algebraic curvature tensors on the ambient space $V$.  Algebraic
curvature tensors are valence four tensors subject to the symmetries of a \name{Riemann}ian curvature tensor.  Furthermore, this
isomorphism is equivariant with respect to the action of the isometry group as a subgroup of $\GLG(V)$.

This is the starting point for the present work:  If the \name{Killing} tensor fields on a constant sectional curvature manifold
correspond bijectively to algebraic curvature tensors, i.\,e.\ simple algebraic objects, then their integrability must be
expressible as a purely algebraic condition on algebraic curvature tensors.  This idea leads finally -- after some tensor
calculus based on results from the theory of representations of the symmetric and general linear groups -- to the following
simple algebraic integrability conditions:

\begin{maintheorem}
	\label{thm:main}
	A \name{Killing} tensor on a constant sectional curvature manifold $M$ is integrable if and only if the associated algebraic
	curvature tensor $R$ on $V$ satisfies the following two conditions:
	\begin{subequations}
		\label{eq:main}
		\begin{align}
			\label{eq:main:1}
			\young(\atwo,\btwo,\ctwo,\dtwo)
			\bar g_{ij}
			R\indices{^i_{b_1\underline a_2\underline b_2}}
			R\indices{^j_{d_1\underline c_2\underline d_2}}
			&=0\\[\smallskipamount]
			\label{eq:main:2}
			\smash[t]{\young(\atwo,\btwo,\ctwo,\dtwo)}
			\young(\aone \bone \cone \done)
			\bar g_{ij}
			\bar g_{kl}
			R\indices{^i_{b_1\underline a_2\underline b_2}}
			R\indices{^j_{a_1}^k_{c_1}}
			R\indices{^l_{d_1\underline c_2\underline d_2}}
			&=0
			\comma
		\end{align}
	\end{subequations}
	where the operators on the left hand side are the \name{Young} symmetrisers for complete antisymmetrisation in the
	(underlined) indices $a_2,b_2,c_2,d_2$ respectively complete symmetrisation in the indices $a_1,b_1,c_1,d_1$.  The tensor
	$\bar g$ denotes the inner product $g$ on $V$ in case $M$ is not flat.  Otherwise, i.\,e.\ if $M\subset V$ is a hyperplane,
	$\bar g$ is the (degenerated) pullback of $g$ via the orthogonal projection $V\to M$.
\end{maintheorem}

This aproach to integrability of \name{Killing} tensor fields on constant sectional curvature manifolds has a certain number of
advantages.
The first and certainly the most important is, that we replace the \name{Nijenhuis} integrability conditions -- a complicated
non-linear system of partial differential equations for a tensor field on a manifold -- by two simple algebraic equations for a
tensor on a vector space.  On the one hand this simplifies a numerical treatment considerably.  Note that integrability can be
checked by a simple evaluation of polynomials of degree two and three.  On the other hand this opens the way for algebraic
methods.  Our formulation for example allowed us to show that the third of the \name{Nijenhuis} integrability conditions is
redundant for \name{Killing} tensors on constant sectional curvature manifolds.  In the special case of \name{Euclid}ean
$3$-space, this result was already mentioned in a footnote of \cite{Horwood&McLenaghan&Smirnov}, stating
	``Steve Czapor (private communication) has simplified the situation considerably.  Using Gröbner basis theory, he has shown
	that (4.4a) and (4.4b) imply (4.4c), for any Killing tensor $\mathbf K\in\mathcal K^2(\mathbb E^3)$.''%
\footnote{%
	Equations (4.4) therein are the \name{Nijenhuis} integrability conditions, c.\,f.\ \eqref{eq:TNS} here.
}

Moreover, we can exhibit a family of integrable \name{Killing} tensors which arises naturally from our algebraic description and
extends a known family \cite{Ibort&Magri&Marmo,Bolsinov&Matveev} which is based on the work of \name{Benenti}
\cite{Benenti92}:
\begin{maincorollary}
	\label{cor:main}
	There exists a family of integrable \name{Killing} tensors on a non-zero constant sectional curvature manifold, given by the
	algebraic curvature tensors
	\begin{align}
		\label{eq:family}
		R&
		=\lambda_2h\varowedge h
		+\lambda_1h\varowedge g
		+\lambda_0g\varowedge g&
		h&\in\Sym^2V&
		\lambda_0,\lambda_1,\lambda_2&\in\mathbb R
		\fullstop
	\end{align}
	Here $\Sym^2V$ is the space of symmetric $2$-tensors\footnote{Without loss of generality $h$ can be supposed trace free.}
	and $\varowedge$ denotes the \name{Kulkarni}-\name{Nomizu} product.  The corresponding \name{Killing} tensors read in local
	coordinates
	\begin{multline*}
		K_{\alpha\beta}
		=2\lambda_2(h_{a_1a_2}h_{b_1b_2}-h_{a_1b_2}h_{b_1a_2})x^{a_1}x^{a_2}\partial_\alpha x^{b_1}\partial_\beta x^{b_2}\\
		+\lambda_1
			(h_{a_1a_2}x^{a_1}x^{a_2}g_{\alpha\beta}-h_{b_1b_2}\partial_\alpha x^{b_1}\partial_\beta x^{b_2})
		+2\lambda_0g_{\alpha\beta}
		\comma
	\end{multline*}
	where the vector components $x^i$ in $V$ are regarded as functions on $M$ by restriction.
\end{maincorollary}
For \name{Euclid}ean $3$-space a complete description of integrable \name{Killing} tensors has been obtained using computer
algebra by \name{Horwood}, \name{McLenaghan} and \name{Smirnov} based on the prior knowledge of the separable coordinate webs,
but a general solution of the integrability conditions has so far been considered intractable \cite{Horwood&McLenaghan&Smirnov}.
Our algebraic equations now render this feasible at least in dimension three.%
\footnote{%
	Note that dimension three of the manifold means dimension four of the ambient vector space.
}
This goes beyond the scope of this article and will be the subject of a forthcoming paper \cite{Schoebel}.

In this context it is noteworthy that the first algebraic integrability condition can be recast into a variety of different
forms.  In terms of the curvature form $\Omega\in\End(V)\otimes\Lambda^2V$ associated ot the algebraic curvature tensor $R$,
condition \eqref{eq:main:1} reads
\[
	\Omega\wedge\Omega=0
	\comma
\]
where the wedge denotes the ususal exterior product in the form component and matrix multiplication in the endomorphism
component.  Another equivalent form, which makes more explicit the index symmetries in terms of
$\GLG(V)$-ir\-re\-pre\-sen\-ta\-tions, is
\begin{equation}
	\label{eq:main:Young}
	\young(\btwo\bone\done,\ctwo,\dtwo,\atwo)
	g_{ij}
	R\indices{^i_{b_1a_2b_2}}
	R\indices{^j_{d_1c_2d_2}}
	=0
	\comma
\end{equation}
where the operator on the left is the \name{Young} symmetriser antisymmetrising first in $a_2,b_2,c_2,d_2$ and then symmetrising
in $b_2,b_1,d_1$.  Similar forms can be obtained for the second algebraic integrability condition \eqref{eq:main:2}.

The second and related advantage of our approach is, that the above algebraic formulation offers new insight into integrability
from the perspectives of representation theory and algebraic geometry as well as geometric invariant theory.  To illustrate
this, regard the solutions of the first integrability condition as the algebraic variety given as the vanishing locus of the
following composed map $\pi\circ\nu$ (where the spaces are denoted for convenience by their corresponding $\GLG(V)$-isomorphism
class):
\pagebreak
\begin{equation}
	\label{eq:zerolocus}
	\begin{array}{rcccl}
		\{\textYF\yng(2,2)\}
		&\xrightarrow{\;\;\nu\;\;}&\Sym^2\{\textYF\yng(2,2)\}
		&\xrightarrow{\;\;\pi\;\;}&\left\{\textYF\yng(3,1,1,1)\right\}\\[\medskipamount]
		\nu\colon R_{a_1b_1a_2b_2}
		&\mapsto&R_{a_1b_1a_2b_2}R_{c_1d_1c_2d_2}\\
		&&\pi\colon T_{a_1b_1a_2b_2c_1d_1c_2d_2}
		&\mapsto&
			\smash[t]{\young(\btwo\bone\done,\ctwo,\dtwo,\atwo)}
			g^{a_1c_1}
			T_{a_1b_1a_2b_2c_1d_1c_2d_2}
			\fullstop
	\end{array}
\end{equation}
The first space is the space of algebraic curvature tensors on $V$ and the second its symmetric product.  The third space is the
image of the \name{Young} symmetriser in $V^{\otimes6}$.
The map $\pi$ is simply a projection given by an index contraction and a projection to an irreducible $\GLG(V)$-representation,
both commuting.  If we pass to the projectivisation of $\pi\circ\nu$,
\[
	\mathbb P(\pi\circ\nu)\colon\;\mathbb P\{\textYF\yng(2,2)\}
	\;\xrightarrow{\;\;\mathbb P\nu\;\;}\;\mathbb P\Sym^2\{\textYF\yng(2,2)\}
	\;\xrightarrow{\;\;\mathbb P\pi\;\;}\;\mathbb P\left\{\textYF\yng(3,1,1,1)\right\}\\[\medskipamount]
	\comma
\]
then the map $\mathbb P\nu$ is nothing else than the \name{Veronese} embedding.  This allows a geometric interpretation of the
first integrability condition's (projectivised) solution space as the intersection of a \name{Veronese} variety with a certain
projective subspace.  The same is true for the second integrability condition.  Of course, every projective variety is
isomorphic to an intersection of a \name{Veronese} variety with a linear space \cite{Harris}, but here all spaces and maps are
given explicitly.

Note also that in our algebraic setting an isometry invariant characterisation of the integrability of \name{Killing} tensors as
in \cite{Horwood&McLenaghan&Smirnov} or \cite{Horwood} reduces to chosing a suitable set of isometry invariants for algebraic
curvature tensors and finding the restrictions imposed on them by the equations \eqref{eq:main}.  This is essentially a problem
in geometric invariant theory.  Due to its importance in general relativity, a variety of such sets have already been proposed
in the case of four-dimensional \name{Lorentz} space.%
\footnote{%
	See \cite{Lim&Carminati.I} and the references therein.
}

Thirdly, we emphasise that our aproach is completely generic in the sense that is does not depend neither on the dimension of
the manifold, nor the value of the constant sectional curvature nor the signature of the pseudo-\name{Riemann}ian metric.  This
becomes manifest in the fact that these data enter the algebraic integrability conditions \eqref{eq:main} only via the signature
and the rank of $\bar g$.  We also remark that our approach is coordinate free.  We do not rely on any particular choice of
coordinates neither on the manifold, nor on the space of \name{Killing} tensors.

Finally, a last but not less important advantage comes from the fact that our algebraic equations are polynomials in a curvature
tensor.  Note that we owe this circumstance to the fact that \name{Killing} tensors on constant sectional curvature manifolds
are described by algebraic curvature tensors and not any other representation space of the isometry group.  This is a rather
fortunate happenstance, because algebraic properties of curvature tensors are extensively studied -- both in differential
geometry as well as in mathematical physics.  Especially the \name{Lorentz}ian case, focus of interest in general relativity, is
important here as it corresponds to hyperbolic space.  In the \name{Riemann}ian case, corresponding to spheres, we even have a
complete classification of the symmetry classes of the \name{Riemann} polynomials appearing in \eqref{eq:main} with respect to
the isometry group $\OG(V,g)\subset\GLG(V)$.  Our methods are inspired by the corresponding article of \name{Fulling} et al.\
\cite{Fulling&King&Wybourne&Cummins}, although we do not rely on results presented there.

\smallskip
We hope that our work will pave the way for an algebraic approach to the study of integrable \name{Killing} tensors and
seperable coordinates.  To this end we list some suggestions for future research based on our results:
\begin{itemize}
	\item \emph{Algebraic interpretation of other families of integrable \name{Killing} tensors,}
		such as the one arising from special conformal \name{Killing} tensors
		\cite{Crampin&Sarlet&Thompson,Crampin&Sarlet01,Crampin03a}, cofactor systems
		\cite{Rauch-Wojciechowski&Marciniak&Lundmark,Lundmark,Rauch-Wojciechowski} or bi-quasi-\name{Hamilton}ian systems
		\cite{Crampin&Sarlet02,Crampin03b} (see also \cite{Benenti05}).  We do not yet fully understand how this --
		geometrically constructed -- family translates to our algebraic framework, but we believe there is a simple algebraic
		interpretation.  Vice versa, we neither know a geometric interpretation of the -- algebraically constructed -- family we
		constructed in the present work.
	\item \emph{An algebraic compatibility condition for the potential.}
		So far we disregarded the compatibility condition for the potential in the \name{Hamilton}-\name{Jacobi} equation.  As
		integrability, it should be expressible entirely in algebraic terms as well.
	\item \emph{Explicit solution of the algebraic integrability conditions.}
		It is possible to solve the algebraic integrability conditions explicitely in dimension three.  This has been done for
		$3$-spheres in \cite{Schoebel} and straightforwardly carries over to \name{Euclid}ean $3$-space.  In higher dimensions
		they can be solved using computer algebra, by means of \name{Gröbner} bases for example.
	\item \emph{Study of the algebraic variety of integrable Killing tensors,}
		defined by the algebraic integrability conditions, especially its dimension.
\end{itemize}

This paper is organised as follows.  In the next section we briefly recall \name{Killing} tensors, constant sectional curvature
manifolds and the notion of integrability in this context.  In \linkref{section}{sec:Benenti} we regard a special family of
integrable \name{Killing} tensors.  This is followed by a short review of some necessary facts from the representation theory of
symmetric and linear groups in \linkref{section}{sec:representationtheory}, which can be skipped by a reader familiar with them.
After that we restate the algebraic characterisation of \name{Killing} tensors in \linkref{section}{sec:keyresult} for our
needs.  \linkref{Section}{sec:K} is the main part, where we derive the algebraic integrability conditions.  As a first
application of them, we extend the family from \linkref{section}{sec:Benenti} in the \hyperref[sec:application]{last section}.

\section{Preliminaries}
\label{sec:preliminaries}

\subsection{\name{Killing} tensors}

Recall that a \dfn{\name{Killing} vector} on a (pseudo-)\name{Riemann}ian manifold $(M,g)$ is a vector field $K^\alpha$ on $M$
satisfying the \name{Killing} equation
\[
	\nabla^{(\alpha}K^{\beta)}=0
\]
where $\nabla$ is \name{Levi}-\name{Civita} connection of $g$ and round brackets denote complete symmetrisation in the enclosed
indices.

\begin{definition}
	A \dfn{\name{Killing} tensor} on $M$ is a symmetric $(2,0)$-tensor field $K^{\alpha\beta}$ satisfying the generalised
	\name{Killing} equation
	\begin{equation}
		\label{eq:Killing}
		\nabla^{(\alpha}K^{\beta\gamma)}=0
		\fullstop
	\end{equation}
\end{definition}

Geometrically, a tensor $K^{\alpha\beta}$ is a \name{Killing} tensor if and only if the amount $K^{\alpha\beta}\dot x_\alpha\dot
x_\beta$ is constant on geodesics.
\pagebreak

\begin{examples}
	\enumskip
	\begin{enumerate}
		\item
			The metric is a \name{Killing} tensor, since it is covariantly constant.
		\item
			As a consequence of the \name{Leibniz} rule the symmetrised tensor product of two \name{Killing} vectors is a
			\name{Killing} tensor.
		\item
			The pullback of a \name{Killing} tensor under an isometry of $M$ is again a \name{Killing} tensor.
	\end{enumerate}
\end{examples}

\subsection{Integrability}

The metric establishes an isomorphism between the tangent space and its dual.  This identifies co- and contravariant tensor
components via lowering or rising indices using the metric.  In particular a \name{Killing} tensor can be identified with a
$(0,2)$-tensor or a $(1,1)$-tensor, the latter being an endomorphism field on $M$.

\begin{definition}
	A $(1,1)$-tensor field $K$ on $M$ is \dfn{integrable}, if almost every point on $M$ admits a neighbourhood with local
	coordinates $x_\alpha$ such that the corresponding coordinate vector fields $\partial_\alpha$ are eigenvector fields of $K$.
\end{definition}

Integrability can be characterised using the \name{Nijenhuis} torsion
\[
	N(X,Y)\coloneq K^2[X,Y]-K[KX,Y]-K[X,KY]+[KX,KY]
	\comma
\]
given in local coordinates by
\begin{equation}
    \label{eq:Nijenhuis}
	N\indices{^\alpha_{\beta\gamma}}
	=K\indices{^\alpha_\delta}\nabla_{[\gamma}K\indices{^\delta_{\beta]}}
	+\nabla_\delta K\indices{^\alpha_{[\gamma}}K\indices{^\delta_{\beta]}}
    \comma
\end{equation}
where square brackets denote complete antisymmetrisation in the enclosed indices \cite{Nijenhuis}:

\begin{theorem}
	Let $N$ be the \name{Nijenhuis} torsion of $K$.  Then $K$ is integrable if and only if the following conditions hold:
	\begin{subequations}
		\label{eq:TNS}
		\begin{align}
			\label{eq:TNS:1}0&=N\indices{^\delta_{[\beta\gamma}}g_{\alpha]\delta}\\
			\label{eq:TNS:2}0&=N\indices{^\delta_{[\beta\gamma}}K_{\alpha]\delta}\\
			\label{eq:TNS:3}0&=N\indices{^\delta_{[\beta\gamma}}K_{\alpha]\varepsilon}K\indices{^\varepsilon_\delta}
		\end{align}
	\end{subequations}
\end{theorem}

\subsection{Manifolds with constant sectional curvature}

The \name{Killing} equation is linear, so the set of \name{Killing} tensors on $M$ is a vector space.  Its maximal dimension is
\begin{align}
	\label{eq:maxdim}
	&\frac{n(n+1)^2(n+2)}{12}
	&n&=\dim M
\end{align}
and will be attained if and only if $M$ has constant sectional curvature \cite{Thompson,Wolf98}.  Every
(pseudo-)\name{Riemann}ian manifold with constant sectional curvature is (up to a rescaling) locally isometric to one of the
standard models below.  This fact allows us to restrict all subsequent considerations to these standard models.

\begin{examples}[Standard models of manifolds with constant sectional curvature]
	\label{ex:models}
	Let $V$ be a vector space of dimension $N\coloneq n+1$, equipped with a non-degenerate inner product $g$ of signature
	$(p,q)$.  Then the (pseudo-)sphere
	\[
		M\coloneq\{x\in V\colon g(x,x)=1\}
	\]
	is an $n$-dimensional (pseudo-)\name{Riemann}ian manifold of constant sectional curvature with respect to the metric
	obtained by restricting $g$ to $M$.  For $(p,q)=(n+1,0)$ this is the standard sphere and for $(p,q)=(n,1)$ the standard
	hyperbolic space.  For other choices of the signature we obtain the different \name{de Sitter} and anti \name{de Sitter}
	spaces.  The isometry group of $M$ is the (pseudo-)orthogonal group $\OG(V,g)\subset\GLG(V)$ acting on $M$ by restriction.

	We can incorporate flat space in this pattern by embedding it as the hyperplane
	\[
		M\coloneq\{x\in V\colon g(x,u)=1\}
	\]
	in $V$ for some fixed normal vector $u\in V$.  The corresponding isometry group is then embedded in $\GLG(V)$ as the semi-direct
	product $M\rtimes\OG(M,g)$.  Although our approach goes through for flat spaces as well, this case often has to be treated
	seperately.
\end{examples}

\section{\name{Benenti} tensors}
\label{sec:Benenti}

For the time being let the inner product $g$ be positive definite so that $M\subset V$ is the unit sphere.  Consider the
diffeomorphism
\begin{equation}
	\label{eq:map}
	\begin{array}{rrcl}
		f_A\colon&M&\to&M\\
		&x&\mapsto&f_A(x)\coloneq\frac{Ax}{\lVert Ax\rVert}
	\end{array}
\end{equation}
for some fixed $A\in\GLG(V)$.  Since $A$ is linear, $f_A$ maps great circles to great circles.  This means that the metric $g$ and
its pullback $g_A\coloneq f_A^*g$ have the same (unparametrised) geodesics, so we can apply the following theorem
\cite{Matveev&Topalov98,Matveev&Topalov00}:

\begin{theorem}
	\label{thm:Matveev-Topalov}
	If two metrics $g$ and $g_A$ on an $n$-dimensional (pseudo-)\name{Riemann}ian manifold have the same unparametrised
	geodesics, then
	\[
		K\coloneq\left(\frac{\det g}{\det g_A}\right)^\frac2{n+1}g_A
	\]
	is a \name{Killing} tensor with respect to $g$.
\end{theorem}

\begin{corollary}
	Let $M\subset V$ be the unit sphere.  Then the group $\GLG(V)$ generates a family of \name{Killing} tensors on $M$ given by
	\begin{equation}
		\label{eq:Benenti}
		K_x(v,w)=g(Ax,Ax)g(Av,Aw)-g(Ax,Av)g(Ax,Aw)
	\end{equation}
	for $v,w\in T_xM$ and $A\in\GLG(V)$.
\end{corollary}

\begin{proof}
	Via the differential of \eqref{eq:map},
	\begin{equation}
		\label{eq:differential}
		(df_A)_xv=\frac{g(Ax,Ax)Av-g(Ax,Av)Ax}{\lVert Ax\rVert^3}
		\comma
	\end{equation}
	one computes the pullback metric $g_A=f_A^*g$ as
	\[
		(g_A)_x(v,w)=\frac{g(Ax,Ax)g(Av,Aw)-g(Ax,Av)g(Ax,Aw)}{\lVert Ax\rVert^4}
		\fullstop
	\]
	To compute its determinant at a point $x$, choose an orthonormal basis $e_1,\ldots,e_n$ of $T_xM$ and extend it with the
	vector $x$ to an orthonormal basis of $V$.  Since $Ax/\lVert Ax\rVert$ is a unit vector normal to $T_{f_A(x)}M$, we have
	\begin{align*}
		\left(\frac{\det g_A}{\det g}\right)_x
		&={\det}^2(df_A)_x
		={\det}^2\left(\tfrac{Ax}{\lVert Ax\rVert},(df_A)_xe_1,\ldots,(df_A)_xe_n\right)\\
		&={\det}^2\left(\tfrac{Ax}{\lVert Ax\rVert},\tfrac{Ae_1}{\lVert Ax\rVert},\ldots,\tfrac{Ae_n}{\lVert Ax\rVert}\right)
		=\left(\frac{\det A}{\lVert Ax\rVert^{n+1}}\right)^2
		\comma
	\end{align*}
	where we used \eqref{eq:differential} for the third equality.  The claim now follows from
	\linkref{theorem}{thm:Matveev-Topalov}, since
	\[
		\left(\frac{\det g}{\det g_A}\right)^\frac2{n+1}g_A
		=
		\left(\frac{\lVert Ax\rVert^{n+1}}{\det A}\right)^\frac4{n+1}
		\frac{g(Ax,Ax)g(Av,Aw)-g(Ax,Av)g(Ax,Aw)}{\lVert Ax\rVert^4}
		\comma
	\]
	differs from \eqref{eq:Benenti} by a constant.
\end{proof}

If $g$ is not positive definite, the map \eqref{eq:map} is not everywhere well defined, but formula \eqref{eq:Benenti} still
gives a well defined \name{Killing} tensor, as we will see in \linkref{section}{sec:keyresult}.  For flat space the result is
analogous, only with more complicated expressions.

\begin{corollary}
	Let $M\subset V$ be a hyperplane with normal $u$.  Then the group $\GLG(V)$ generates a family of \name{Killing} tensors on
	$M$ given by
	\begin{multline}
		\label{eq:Benenti:flat}
		K_x(v,w)=g(Ax,u)g(Ax,u)g(Av,Aw)-g(Ax,u)g(Ax,Av)g(Aw,u)\\-g(Ax,u)g(Ax,Aw)g(Av,u)+g(Ax,Ax)g(Av,u)g(Aw,u)
		\fullstop
	\end{multline}
\end{corollary}
\begin{proof}
	The proof follows the lines of the proof in the non-flat case, considering the diffeomorphism
	\begin{equation}
		\label{eq:map:flat}
		\begin{array}{rrcl}
			f_A\colon&M&\to&M\\
			&x&\mapsto&f_A(x)\coloneq\frac{Ax}{g(u,Ax)}
		\end{array}
	\end{equation}
	instead of \eqref{eq:map}.  This map is line preserving, so we can again apply \linkref{theorem}{thm:Matveev-Topalov}.  The
	differential
	\[
		(df_A)_xv=\frac{g(u,Ax)Av-g(u,Av)Ax}{g(u,Ax)^2}
	\]
	of \eqref{eq:map:flat} yields the pullback metric
	\begin{multline*}
		(g_A)_x(v,w)\\
		=\frac1{g(u,Ax)^4}\Bigl[
			 g( u,Ax)g( u,Ax)g(Av,Aw)
			-g( u,Av)g( u,Ax)g(Ax,Aw)\\
			-g( u,Ax)g( u,Aw)g(Av,Ax)
			+g( u,Av)g( u,Aw)g(Ax,Ax)
		\Bigr]
		\fullstop
	\end{multline*}
	As above, we compute the determinant of $(g_A)_x$ using an othonormal basis $e_1,\ldots,e_n$ of $T_xM$,
	\begin{align*}
		\left(\frac{\det g_A}{\det g}\right)_x
		&={\det}^2(df)_x
		={\det}^2\bigl(u,(df)_xe_1,\ldots,(df)_xe_n\bigr)\\
		&={\det}^2\bigl(\tfrac{Ax}{g(Ax,u)},(df)_xe_1,\ldots,(df)_xe_n\bigr)\\
		&={\det}^2\left(\tfrac{Ax}{g(Ax,u)},\tfrac{Ae_1}{g(Ax,u)},\ldots,\tfrac{Ae_n}{g(Ax,u)}\right)\\
		&=\left(\frac{\det A}{g(Ax,u)^{n+1}}\,\det\bigl(x,e_1,\ldots,e_n\bigr)\right)^2
		=\left(\frac{\det A}{g(Ax,u)^{n+1}}\right)^2
		\comma
	\end{align*}
	and the corollary follows from \linkref{theorem}{thm:Matveev-Topalov}.
\end{proof}

\name{Killing} tensors of type \eqref{eq:Benenti} respectively \eqref{eq:Benenti:flat} coincide with those introduced in local
coordinates in \cite{Benenti92}.  Following \cite{Ibort&Magri&Marmo,Bolsinov&Matveev} we will call them \dfn{\name{Benenti}
tensors}.

\section{Some facts from representation theory}
\label{sec:representationtheory}

We briefly collect some facts we need from the representation theory of symmetric and general linear groups.  More details can
be found in standard textbooks.

\subsection{\name{Young} symmetrisers and the \name{Littlewood}-\name{Richardson} rule}

The isomorphism classes of irreducible representations (``irreps'') of $\SG_d$ are labelled by \dfn{partitions} of $d$, i.\,e.\
integers $d_1\ge d_2\ge\ldots\ge d_r>0$ with $d_1+\ldots+d_r=d$.  It is useful to depict partitions as so called
\dfn{\name{Young} frames}, as in the following example:

\begin{example}
	\label{ex:partitions}
	The partitions of $4$ are $4=4$, $3+1=4$, $2+2=4$, $2+1+1=4$ and $1+1+1+1=4$ with corresponding \name{Young} frames
	\[
		\yng(4)\qquad
		\yng(3,1)\qquad
		\yng(2,2)\qquad
		\yng(2,1,1)\qquad
		\yng(1,1,1,1)
		\fullstop
	\]
\end{example}

The dimension of an irrep corresponding to a \name{Young} frame is given by dividing $d!$ by the product of the hook lengths of
all boxes of the frame, where the hook length of a box is
\[
	(\text{the number of boxes to the right})\;+\;1\;+\;(\text{the number of boxes below})
	\fullstop
\]
This is the so called \emph{hook formula}.

\begin{example}
	The hook lengths of the boxes of the \name{Young} frames in \linkref{example}{ex:partitions} are
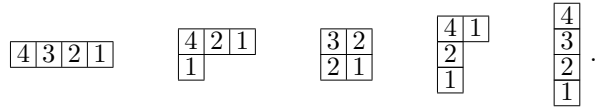
	\[
		\young(4321)\qquad
		\young(421,1)\qquad
		\young(32,21)\qquad
		\young(41,2,1)\qquad
		\young(4,3,2,1)
		\fullstop
	\]
	The corresponding dimensions are thus $1$, $3$, $2$, $3$, $1$.
\end{example}

The irreps of $\SG_d$ can be realised on subspaces of the \dfn{group algebra}
\[
	\R\SG_d=\left\{\sum_{\pi\in\SG_d}\lambda_\pi\pi\colon\lambda_\pi\in\R\right\}
	\fullstop
\]
This is the free real vector space over $\SG_d$ as a set, endowed with the obvious multiplication given by extending the group
multiplication in $\SG_d$ linearly.  Multiplication with elements of $\SG_d$ from the left defines a representation of $\SG_d$
on $\R\SG_d$.  One then constructs projectors onto irreducible subspaces as follows.

Let $\tau$ be a \name{Young} frame whose $d$ boxes are filled with the integers $1,\ldots,d$ without repetition in an arbitrary
order.  We call this a \dfn{\name{Young} tableau}.  For each row $r$ of $\tau$ let $\SG_r\subset\SG_d$ be the symmetric group of
permutations of the labels in $r$ and similarly $\SG_c\subset\SG_d$ for each column $c$ of $\tau$.  The \dfn{\name{Young}
symmetriser} corresponding to $\tau$ is then the element in $\R\SG_d$ defined by
\begin{equation}
	\label{eq:Young_symmetriser}
	\prod_{\text{$r$ row    of $\tau$}}\mspace{-16mu}\underbrace{\Bigl(\;\sum_{\pi\in\SG_r}          \pi\;\Bigr)}_{\text{row        symmetriser of $r$}}\quad
	\prod_{\text{$c$ column of $\tau$}}\mspace{-16mu}\underbrace{\Bigl(\;\sum_{\pi\in\SG_c}(\sign\pi)\pi\;\Bigr)}_{\text{column antisymmetriser of $c$}}
	\fullstop
\end{equation}
We will often identify a \name{Young} tableau with its corresponding \name{Young} symmetriser as in the following example.
\begin{example}
	\begin{align}
		\label{eq:1234}
		\young(12,34)
		&=\young(12)\young(34)\young(1,3)\young(2,4)
		=\bigl(e+(12)\bigr)\bigl(e+(34)\bigr)\bigl(e-(13)\bigr)\bigl(e-(24)\bigr)\\
		\notag&=e+(12)+(34)-(13)-(24)+(12)(34)+(13)(24)-(132)-(234)\\
		\notag&\qquad\qquad-(124)-(143)+(1423)+(1324)-(1234)-(1432)+(14)(23)
	\end{align}
\end{example}

If we denote by $\hook\tau$ the product of the hook lengths of all boxes in a \name{Young} tableau $\tau$, the corresponding
\name{Young} symmetriser satisfies
\begin{equation}
	\label{eq:square}
	\tau^2=\hook\tau\tau
	\fullstop
\end{equation}
Every element in $\R\SG_d$ is at the same time a linear operator on $\R\SG_d$ via multiplication from the right.  Then the image
of $\tau$ is an irreducible subrepresentation of $\R\SG_d$ whose isomorphism class is given by the \name{Young} frame of $\tau$.
Rewriting \eqref{eq:square} we get a projector onto the corresponding subspace:
\begin{equation}
	\label{eq:Young_projector}
	\left(\frac\tau{\hook\tau}\right)^2=\frac\tau{\hook\tau}
	\fullstop
\end{equation}

The group algebra $\R\SG_d$ carries a natural $\SG_d$-invariant inner product defined by taking the basis $\SG_d$ to be
orthonormal.  With respect to this inner product the adjoint of an element in $\R\SG_d$ is
\[
	\Bigl(\,\sum_{\pi\in\SG_d}\alpha_\pi\pi     \Bigr)^\adjoint
	=       \sum_{\pi\in\SG_d}\alpha_\pi\pi^{-1}
	\fullstop
\]
Note that the column symmetrisers of a \name{Young} tableau are self-adjoint and commute and likewise for its row
antisymmetrisers.  Taking the adjoint of a \name{Young} symmetriser therefore simply exchanges the two products in
\eqref{eq:Young_symmetriser}.
\begin{example}
	\begin{equation}
		\label{eq:1234*}
		{\young(12,34)}^\adjoint
		=\young(1,3)\young(2,4)\young(12)\young(34)
		=\bigl(e-(13)\bigr)\bigl(e-(24)\bigr)\bigl(e+(12)\bigr)\bigl(e+(34)\bigr)\\
	\end{equation}
\end{example}
We see that \name{Young} symmetrisers are in general not self-adjoint, so that the corresponding \name{Young} projectors
\eqref{eq:Young_projector} will not be orthogonal.  However, a \name{Young} projector and its adjoint project onto isomorphic
irreps \cite{Fulton}.  From \eqref{eq:Young_symmetriser} and \eqref{eq:square} follows that up to an apropriate factor the
element $\tau\tau^\adjoint$ is an orthogonal projector onto the image of $\tau$ and similarly $\tau^\adjoint\tau$ onto the
image of $\tau^\adjoint$.

\bigskip

Two representations $\lambda_1$ and $\lambda_2$ of $\SG_{d_1}$ on $V_1$ respectively $\SG_{d_2}$ on $V_2$ determine a
representation of $\SG_{d_1}\times\SG_{d_2}$ on $V_1\otimes V_2$ given by
\begin{align*}
	(g_1\times g_2)(v_1\otimes v_2)&\coloneq(g_1v_1)\otimes(g_2v_2)&
	g_1\in\SG_{d_1},&\;\;
	g_2\in\SG_{d_2}
	\fullstop
\end{align*}
Via the inclusion $\SG_{d_1}\times\SG_{d_2}\hookrightarrow\SG_{d_1+d_2}$ this induces a representation
$\lambda_1\boxtimes\lambda_2$ of $\SG_{d_1+d_2}$ on $V_1\otimes V_2$, called the exterior tensor product of $\lambda_1$ and
$\lambda_2$.  The \name{Littlewood}-\name{Richardson} rule tells us how this product decomposes into irreps:

\begin{theorem}[The \name{Littlewood}-\name{Richardson} rule]
	The decomposition of the exterior tensor product $\lambda_1\boxtimes\lambda_2$ of two irreps $\lambda_1$ of $S_{d_1}$ and
	$\lambda_2$ of $S_{d_2}$ into irreps of $S_{d_1+d_2}$ is given by the following algorithm.  First label all the boxes in
	$\lambda_2$ with their corresponding row number.  Then add the labelled boxes of $\lambda_2$ to $\lambda_1$ -- one by one
	from top to bottom -- respecting at each step the following rules:
\pagebreak
	\begin{enumerate}
		\item
			The obtained frame is a legitimate \name{Young} frame.
		\item
			No two boxes in the same column are labelled equally.
		\item
			If the labels are read off from right to left along the rows from top to bottom, one never encounters more 1's than
			2's, and so on.
	\end{enumerate}
	Each of the distinct \name{Young} frames constructed in this way specifies an irreducible sum term in the decomposition of
	$\lambda_1\boxtimes\lambda_2$ with the correpsonding multiplicity, since the same shaped \name{Young} frame may arise in more
	than one way.  Since the exterior tensor product is commutative, one can choose the simpler \name{Young} frame for $\lambda_2$.
\end{theorem}

\begin{example}
	\begin{align}
		\label{eq:2x2}
		\yng(2)\boxtimes\young(11)
		&\isom
		\young(~~11)\oplus\young(~~1,1)\oplus\young(~~,11)\\
		\label{eq:11x11}
		\young(~,~)\boxtimes\young(1,2)
		&\isom
		\young(~1,~2)\oplus\young(~1,~,2)\oplus\young(~,~,1,2)\\
		\label{eq:111x3}
		\yng(1,1,1)\boxtimes\young(111)
		&\isom
		\young(~111,~,~)\oplus\young(~11,~,~,1)
	\end{align}
\end{example}

\subsection{\name{Weyl}'s construction and algebraic curvature tensors}

Every \name{Young} tableau $\tau$ gives rise to a $\GLG(V)$-irrep in the following way, called \name{Weyl}'s construction.
Consider the $d$-fold tensor product $V^{\otimes d}$ as a representation space for both $\GLG(V)$ and $\SG_d$ with respect to
the commuting actions
\[
	\begin{array}{r@{\,}c@{\;}c@{\;}c@{\,}l@{\;}c@{\;}r@{\,}c@{\;}c@{\;}c@{\,}l@{\qquad}r@{\;}c@{\;}l}
		g(v_{i_1}&\otimes&\ldots&\otimes&v_{i_d})&\coloneq&(gv_{i_1})&\otimes&\ldots&\otimes&(gv_{i_d})&g&\in&\GLG(V)\\
		\pi(v_1&\otimes&\ldots&\otimes&v_d)&\coloneq&v_{\pi^{-1}(1)}&\otimes&\ldots&\otimes&v_{\pi^{-1}(d)}&\pi&\in&\SG_d
		\fullstop
	\end{array}
\]
Each element in $\R\SG_d$ gives a linear operator on $V^{\otimes d}$ by linearly extending the action of $\SG_d$.  The image of
a \name{Young} symmetriser $\tau\in\R\SG_d$ is then an irreducible $\GLG(V)$-subrepresentation.  Considering instead the dual
action of $\GLG(V)$ on the dual $\bar V$ of $V$ yields the (non-isomorphic) dual representation on $\bar V^{\otimes d}$.
\begin{example}
	The \name{Young} symmetriser \eqref{eq:1234} determines the following operator on $\bar V^{\otimes4}$ whose image
	constitutes a $\GLG(V)$-irrep:
	\begin{equation}
		\label{eq:symprojector}
		\begin{array}{r@{\;}c@{\;}c@{\;}c@{\;}c}
			{\young(\aone\atwo,\bone\btwo)}\phantom{^\adjoint}T_{a_1a_2b_1b_2}
			=&T_{a_1a_2b_1b_2}
			-&T_{a_2a_1b_1b_2}
			-&T_{a_1a_2b_2b_1}
			+&T_{a_2a_1b_2b_1}\\
			+&T_{b_1a_2a_1b_2}
			-&T_{a_2b_1a_1b_2}
			-&T_{b_1a_2b_2a_1}
			+&T_{a_2b_1b_2a_1}\\
			+&T_{a_1b_2b_1a_2}
			-&T_{b_2a_1b_1a_2}
			-&T_{a_1b_2a_2b_1}
			+&T_{b_2a_1a_2b_1}\\
			+&T_{b_1b_2a_1a_2}
			-&T_{b_2b_1a_1a_2}
			-&T_{b_1b_2a_2a_1}
			+&T_{b_2b_1a_2a_1}
			\fullstop
		\end{array}
	\end{equation}
	Note that on the level of tensor components one gets the correct action of $\SG_d$ by permuting index names, not index
	positions.
\end{example}
In the same way we can construct an irreducible $\GLG(V)$-subrepresentation from the adjoint $\tau^\adjoint$ of a \name{Young}
tableau $\tau$.
\begin{example}
	\begin{equation}
		\label{eq:antiprojector}
		\begin{array}{r@{\;}c@{\;}c@{\;}c@{\;}c}
			{\young(\aone\atwo,\bone\btwo)}^\adjoint T_{a_1b_1a_2b_2}
			=&T_{a_1b_1a_2b_2}
			+&T_{a_2b_1a_1b_2}
			+&T_{a_1b_2a_2b_1}
			+&T_{a_2b_2a_1b_1}\\
			-&T_{b_1a_1a_2b_2}
			-&T_{a_2a_1b_1b_2}
			-&T_{b_1b_2a_2a_1}
			-&T_{a_2b_2b_1a_1}\\
			-&T_{a_1b_1b_2a_2}
			-&T_{b_2b_1a_1a_2}
			-&T_{a_1a_2b_2b_1}
			-&T_{b_2a_2a_1b_1}\\
			+&T_{b_1a_1b_2a_2}
			+&T_{b_2a_1b_1a_2}
			+&T_{b_1a_2b_2a_1}
			+&T_{b_2a_2b_1a_1}
			\fullstop
		\end{array}
	\end{equation}
\end{example}
\begin{example}[Algebraic curvature tensors]
	An \dfn{algebraic curvature tensor} on $V$ is an element $R\in\bar V^{\otimes4}$ satisfying the symmetry relations of a
	\name{Riemann}ian curvature tensor, i.e.:
	\begin{subequations}
		\label{eq:R}
		\begin{align}
			\label{eq:R:anti}\text{antisymmetry:\quad}&R_{b_1a_1a_2b_2}=-R_{a_1b_1a_2b_2}=R_{a_1b_1b_2a_2}\\
			\label{eq:R:pair}\text{pair symmetry:\quad}&R_{a_2b_2a_1b_1}=R_{a_1b_1a_2b_2}\\
			\label{eq:R:Bianchi}\text{\name{Bianchi} identity:\quad}&R_{a_1b_1a_2b_2}+R_{a_1a_2b_2b_1}+R_{a_1b_2b_1a_2}=0
		\end{align}
	\end{subequations}
	A little computation shows, that on one hand any tensor of the form \eqref{eq:antiprojector} has these symmetries and that
	on the other hand any tensor having these symmetries verifies
	\[
		\frac1{12}{\young(\aone\atwo,\bone\btwo)}^\adjoint R_{a_1b_1a_2b_2}=R_{a_1b_1a_2b_2}
		\fullstop
	\]
	This means that algebraic curvature tensors form an irreducible $\GLG(V)$-re\-pre\-sen\-ta\-tion space.
\end{example}
\begin{example}[Symmetrised algebraic curvature tensors]
	A \dfn{symmetrised algebraic curvature tensor} on $V$ is an element $S\in\bar V^{\otimes4}$ satisfying the following
	symmetry relations:
	\begin{subequations}
		\label{eq:S}
		\begin{align}
			\label{eq:S:sym}\text{symmetry:\quad}&S_{a_2a_1b_1b_2}=+S_{a_1a_2b_1b_2}=S_{a_1a_2b_2b_1}\\
			\label{eq:S:pair}\text{pair symmetry:\quad}&S_{b_1b_2a_1a_2}=S_{a_1a_2b_1b_2}\\
			\label{eq:S:Bianchi}\text{\name{Bianchi} identity:\quad}&S_{a_1a_2b_1b_2}+S_{a_1b_2a_2b_1}+S_{a_1b_1b_2a_2}=0
		\end{align}
	\end{subequations}
	As in the previous example, this is equivalent to
	\[
		\frac1{12}{\young(\aone\atwo,\bone\btwo)}S_{a_1a_2b_1b_2}=S_{a_1a_2b_1b_2}
	\]
	so that symmetrised algebraic curvature tensors form another irreducible $\GLG(V)$-representation space.
\end{example}
\begin{remark}[\name{Bianchi} identity]
	In presence of the first two symmetries, there are several equivalent forms of the \name{Bianchi} identity in both cases.
	First, we can write it as vanishing cyclic sum over \emph{any} three of the four indices, for example as
	$R_{a_1b_1a_2b_2}+R_{b_1a_2a_1b_2}+R_{a_2a_1b_1b_2}=0$.  Second, for (symmetrised) algebraic curvature tensors the
	\name{Bianchi} identity is equivalent to the vanishing of the complete antisymmetrisation (symmetrisation) in any three of
	the four indices, for example to
	\begin{align}
		\notag
		\smash[b]{\young(\bone,\atwo,\btwo)}R_{a_1b_1a_2b_2}&=0\\
		\intertext{or}
		\label{eq:Bianchi:alt}
		\young(\atwo\bone\btwo)S_{a_1a_2b_1b_2}&=0
		\fullstop
	\end{align}
	In the following we will refer to all these forms as ``\name{Bianchi} identity''.
\end{remark}
The $\GLG(V)$-irreps constructed from $\tau$ and $\tau^\adjoint$ are isomorphic.
\begin{example}
	An explicit isomorphism between the irreps of $\GLG(V)$ on algebraic curvature tensors respectively on symmetrised algebraic
	curvature tensors is given by
	\begin{subequations}
		\label{eq:R<->S}
		\begin{align}
			\label{eq:R->S}
			S_{a_1a_2b_1b_2}&=\frac1{\sqrt3}\bigl(R_{a_1b_1a_2b_2}+R_{a_1b_2a_2b_1}\bigr)\\
			\label{eq:S->R}
			R_{a_1b_1a_2b_2}&=\frac1{\sqrt3}\bigl(S_{a_1a_2b_1b_2}-S_{a_1b_2b_1a_2}\bigr)
			\comma
		\end{align}
	\end{subequations}
	which is easily checked using the symmetries \eqref{eq:R} and \eqref{eq:S}.
\end{example}

Two \name{Young} tableaux determine isomorphic $\GLG(V)$-representations if and only if they fill the same \name{Young} frame
$\lambda$.  Their dimension can be computed by labelling each box of $\lambda$ with
\[
	(\text{number of the box' column})\;+\;N\;-\,(\text{number of the box' row})
	\comma
\]
taking the product of these labels and dividing by the product $\hook\lambda$ of the hook lengths of $\lambda$.  It is standard
to denote the isomorphism class obtained from $\lambda$ via the \name{Weyl} construcction by $\{\lambda\}$.

\begin{example}
	The isomorphism class of the irrepresentations given by (symmetrised) algebraic curvature tensors is
	$\{\textYF\yng(2,2)\}$ an has dimension
	\begin{equation}
	    \label{eq:dim(22)}
		\dim\;\Bigl\{\YF\yng(2,2)\Bigr\}
		=\frac{(N-1)N^2(N+1)}{12}
		\fullstop
	\end{equation}
\end{example}

The dual pairing between $\bar V^{\otimes d}$ and $V^{\otimes d}$ is given by index contraction.  From the identity
\[
	S_{i_1\cdots i_d}T^{i_{\pi(1)}\cdots i_{\pi(d)}}
	=S_{i_{\pi^{-1}(1)}\cdots i_{\pi^{-1}(d)}}T^{i_1\cdots i_d}
\]
we deduce
\[
	S_{i_1\cdots i_d}\bigl(\pi T^{i_1\cdots i_d}\bigr)
	=\bigl(\pi^{-1}     S_{i_1\cdots i_d}\bigr)T^{i_1\cdots i_d}
	=\bigl(\pi^\adjoint S_{i_1\cdots i_d}\bigr)T^{i_1\cdots i_d}
	\fullstop
\]
This means that with respect to the dual pairing the adjoint of the linear operator on $V^{\otimes d}$ given by an element
$\tau\in\R\SG_d$ acting on upper indices is given by $\tau^\adjoint$ acting on lower indices.  To save notation we will use
parentheses as above to indicate whether a given element of $\R\SG_d$ acts on upper or lower indices.

\section{An algebraic characterisation of \name{Killing} tensors}
\label{sec:keyresult}

Recall that we consider standard models of constant sectional curvature manifolds $M$, embedded isometrically as hypersurfaces
in a \name{Euclid}ean vector space $(V,g)$.  As common in relativity, we distinguish coordinates on $M$ and $V$ by index types:

\begin{convention}
	Throughout this exposition we use latin indices $a,b,c,\ldots$ for components in $V$ (ranging from $0$ to $n$) and greek
	indices $\alpha,\beta,\gamma,\ldots$ for local coordinates on $M$ (ranging from $1$ to $n$).  We can then denote both the
	inner product on $V$ as well as the induced metric on $M$ by the same letter $g$ and distinguish both via the type of
	indices.  Consequently, latin indices are rised and lowered using $g_{ab}$ and greek ones using $g_{\alpha\beta}$.
\end{convention}

The key result for our algebraic characterisation of integrability stems from \name{Mc\-Le\-na\-ghan}, \name{Mil\-son} \&
\name{Smir\-nov} and is a special case of theorem~3.5 in \cite{McLenaghan&Milson&Smirnov}:

\begin{theorem}
	\label{thm:keyresult}
	Let $M\subset V$ be one of the standard models for constant sectional curvature manifolds as in \linkref{example}{ex:models}.
	\begin{enumerate}
		\item
			There is an isomorphism between the irreducible $\GLG(V)$-representation space of antisymmetric tensors $A_{ab}$ on
			$V$ and the vector space of \name{Killing} vectors $K$ on $M$, given by
			\begin{align*}
				K_x(v)&\coloneq A_{ab}x^av^b&
				x\in M,&\;\;
				v\in T_xM
				\comma
			\end{align*}
			when $K$ is written covariantly.
		\item
			There is an isomorphism between the irreducible $\GLG(V)$-representation space of algebraic curvature tensors
			$R_{a_1b_1a_2b_2}$ on $V$ and the vector space of \name{Killing} tensors $K$ on $M$, given by
			\begin{align}
				\label{eq:keyiso:anti}
				K_x(v,w)&\coloneq R_{a_1b_1a_2b_2}x^{a_1}v^{b_1}x^{a_2}w^{b_2}&
				x\in M,&\;\;
				v,w\in T_xM
				\comma
			\end{align}
			when $K$ is written covariantly.
	\end{enumerate}
	Both isomorphisms are equivariant with respect to the action of the isometry group of $M$ as a subgroup of $\GLG(V)$.
\end{theorem}

First note that due to the term $x^{a_1}x^{a_2}$ the tensor $R_{a_1b_1a_2b_2}$ in \eqref{eq:keyiso:anti} is implicitely
symmetrised in the indices $a_1,a_2$ and can therefore be replaced by the corresponding symmetrised algebraic curvature tensor
\eqref{eq:R->S}.  Since this will simplify subsequent computations considerably, we reformulate the the second part of the
theorem:

\begin{corollary}
	\label{cor:keyiso:sym}
	Let $M\subset V$ be one of the standard models as in \linkref{example}{ex:models}.  Then
	\begin{align}
		\label{eq:keyiso:sym}
		K(v,w)&\coloneq S_{a_1a_2b_1b_2}x^{a_1}x^{a_2}v^{b_1}w^{b_2}&
		x\in M,&\;\;
		v,w\in T_xM
	\end{align}
	defines an isomorphism between the irreducible $\GLG(V)$-representation space of symmetrised algebraic curvature tensors
	$S_{a_1a_2b_1b_2}$ and the vector space of \name{Killing} tensors $K$ on $M$, which is equivariant with respect to the
	action of the isometry group.
\end{corollary}
We include a short proof here, because some ideas and intermediate results will be useful in subsequent computations.
\begin{remark}
	If we consider the standard coordinates $x^a$ of a vector $x\in V$ as functions on $M\subset V$ by restriction, then we can
	write for any tangent vector $u\in T_xM\subset V$ with coordinates $u^a$ in $V$:
	\begin{equation}
		\label{eq:nabla}
		\nabla_ux^a=u^a
		\comma
	\end{equation}
	where $\nabla$ denotes the \name{Levi}-\name{Civita} connection of the metric on $M$.
\end{remark}
\begin{proof}[Proof (of \linkref{corollary}{cor:keyiso:sym})]
	We first show that the tensor \eqref{eq:keyiso:sym} actually is a \name{Killing} tensor.  Extend the vectors $u,v,w\in T_xM$
	to arbitrary vector fields $\bar u,\bar v,\bar w$ on $M$.  Using \eqref{eq:keyiso:sym} and \eqref{eq:nabla}, we compute
	\begin{equation}
		\label{eq:gradient}
		\begin{split}
			\nabla_uK(v,w)
			&
			=\nabla_{\bar u}\bigl(K(\bar v,\bar w)\bigr)
			-K\bigl(\nabla_{\bar u}\bar v,\bar w\bigr)
			-K\bigl(\bar v,\nabla_{\bar u}\bar w\bigr)
			\\&
			=S_{a_1a_2b_1b_2}
				\Bigl(
					u^{a_1}x^{a_2}v^{b_1}w^{b_2}
					+x^{a_1}u^{a_2}v^{b_1}w^{b_2}
				\\&\mspace{128mu}
					+x^{a_1}x^{a_2}(\nabla_{\bar u}\nabla_{\bar v}-\nabla_{\nabla_{\bar u}\bar v})x^{b_1}w^{b_2}
				\\&\mspace{128mu}
					+x^{a_1}x^{a_2}v^{b_1}(\nabla_{\bar u}\nabla_{\bar w}-\nabla_{\nabla_{\bar u}\bar w})x^{b_2}
				\Bigr)
			\fullstop
		\end{split}
	\end{equation}
	The operator $H(u,v)=\nabla_{\bar u}\nabla_{\bar v}-\nabla_{\nabla_{\bar u}\bar v}$ is the \name{Hesse} operator and does
	not depend on the extensions $\bar u$ and $\bar v$ of $u$ and $v$.
	\begin{lemma}
		The \name{Hesse} form of the function $x^b$ is given by
		\[
			H(u,v)x^b=-g(v,w)x^b
		\]
		if $M$ is not flat and zero otherwise.
	\end{lemma}
	\begin{proof}
		Extend the vector fields $\bar u,\bar v$ on $M$ further to all of $V$ and denote the \name{Levi}-\name{Civita}
		connection in $V$ by $\bar\nabla$.  Then, using \eqref{eq:nabla},
		\begin{align*}
			H(u,v)x^a
			&
			=\nabla_{\bar u}\nabla_{\bar v}x^a-\nabla_{\nabla_{\bar u}\bar v}x^a
			=\nabla_{\bar u}(\bar v^a)-\bigl(\nabla_{\bar u}\bar v\bigr)^a
			=\bar\nabla_{\bar u}(\bar v^a)-\bigl(\nabla_{\bar u}\bar v\bigr)^a
			\\&
			=\bigl[\bar\nabla_{\bar u}\bar v-\nabla_{\bar u}\bar v\bigr]^a
			=\bigl[\operatorname{II}(u,v)\bigr]^a
		\end{align*}
		It is not difficult to show that the second fundamental form of $M\subset V$ is $\operatorname{II}_x(u,v)=-g(u,v)x$ if
		$M$ is not flat.  Otherwise the lemma is trivial.
	\end{proof}
	We resume the proof of \linkref{corollary}{cor:keyiso:sym}.  Together with the \name{Bianchi} identity the lemma shows that
	the terms in \eqref{eq:gradient} containing the \name{Hesse} form $H(u,v)x^b$ vanish.  Using the symmetry of
	$S_{a_1a_2b_1b_2}$ in $a_1$,$a_2$ we get
	\[
		\nabla_uK(v,w)=2S_{a_1a_2b_1b_2}x^{a_1}u^{a_2}v^{b_1}w^{b_2}
		\fullstop
	\]
	We reformulate the results obtained so far in local coordinates, using \eqref{eq:nabla} again:
	\begin{subequations}
		\label{eq:KK'}
		\begin{align}
			\label{eq:K}
			K_{\alpha\beta}&=S_{a_1a_2b_1b_2}x^{a_1}x^{a_2}\nabla_\alpha x^{b_1}\nabla_\beta x^{b_2}\\
			\label{eq:K'}
			\nabla_\gamma K_{\alpha\beta}&=2S_{c_1c_2d_1d_2}x^{c_1}\nabla_\gamma x^{c_2}\nabla_\alpha x^{d_1}\nabla_\beta x^{d_2}
		\end{align}
	\end{subequations}
	That $K$ satisfies the \name{Killing} equation \eqref{eq:Killing} is now a direct consequence of the \name{Bianchi}
	identity.

	We continue the proof by showing that the map defined by \eqref{eq:keyiso:sym} is injective.  Suppose
	\begin{equation}
		\label{eq:injectivity}
		S_{a_1a_2b_1b_2}x^{a_1}x^{a_2}v^{b_1}w^{b_2}=0
	\end{equation}
	for all $x,v,w\in V$ with $x\in M$ and $v,w\in T_xM$.  From the \name{Bianchi} identity we see that \eqref{eq:injectivity}
	is trivially satisfied if $v=x$ or $w=x$.  We can thus drop the condition $v,w\in T_xM$ by decomposing $v,w\in V$ according
	to the splitting $V=T_xM\oplus\R x$.  We can also drop the condition $x\in M$, since $\R M$ is dense in $V$.  Finally, by
	polarisation we obtain $S_{a_1a_2b_1b_2}x^{a_1}y^{a_2}v^{b_1}w^{b_2}=0$ for all $x,y,v,w\in V$ which is equivalent to $S=0$.

	Bijectivity now follows from the fact that the dimensions \eqref{eq:dim(22)} and \eqref{eq:maxdim} of both spaces coincide
	for $N=n+1$.  Equivariance is evident.
\end{proof}
\begin{definition}
	The \dfn{\name{Kulkarni}-\name{Nomizu} product} $h\varowedge k$ of two symmetric tensors $h$ and $k$ is the algebraic
	curvature tensor
	\begin{align*}
		(h\varowedge k)_{a_1b_1a_2b_2}
		\coloneq\;&h_{a_1a_2}k_{b_1b_2}-h_{a_1b_2}k_{b_1a_2}-h_{b_1a_2}k_{a_1b_2}+h_{b_1b_2}k_{a_1a_2}\\
		=\;&\frac14{\young(\aone\atwo,\bone\btwo)}^\adjoint h_{a_1a_1}k_{b_1b_2}
		\fullstop
	\end{align*}
	In the language of representation theory this product corresponds to the projection of \eqref{eq:2x2} to the
	$\textYF\yng(2,2)$-component.  
\end{definition}
\begin{example}[The metric]
	\label{ex:metric}
	If $M$ is not flat, the metric as a \name{Killing} tensor on $M$ is represented by the algebraic curvature tensor
	$\tfrac12g\varowedge g$.  This follows from \eqref{eq:keyiso:anti} and 
	\begin{equation}
		\label{eq:metric}
		\bigl(\tfrac12g\varowedge g\bigr)_{a_1b_1a_2b_2}
		=\frac18{\young(\aone\atwo,\bone\btwo)}^\adjoint g_{a_1a_2}g_{b_1b_2}
		=g_{a_1a_2}g_{b_1b_2}-g_{a_1b_2}g_{a_2b_1}
		\comma
	\end{equation}
	since $g_{a_1a_2}x^{a_1}x^{a_2}=1$ and $g_{a_1b_2}x^{a_1}v^{b_2}=0$.

	In the flat case the metric is represented by $(u\otimes u)\varowedge g$, given by
	\begin{equation}
		\label{eq:metric:flat}
		\begin{split}
			\bigl((u\otimes u)\varowedge g\bigr)_{a_1b_1a_2b_2}
			&=\frac14{\young(\aone\atwo,\bone\btwo)}^\adjoint u_{a_1}u_{a_2}g_{b_1b_2}\\
			&=u_{a_1}u_{a_2}g_{b_1b_2}
			-u_{a_1}u_{b_2}g_{b_1a_2}
			-u_{b_1}u_{a_2}g_{a_1b_2}
			-u_{b_1}u_{b_2}g_{a_1a_2}
			\comma
		\end{split}
	\end{equation}
	since $u_ax^a=1$ and $u_bv^b=0$.
\end{example}
The isometry group of $M$ is a subgroup of $\GLG(V)$.  As a consequence of \linkref{theorem}{thm:keyresult} its action on the
space of \name{Killing} tensors extends to a natural action of $\GLG(V)$.
\begin{example}[\name{Benenti} tensors]
	Rewriting \eqref{eq:Benenti} respectively \eqref{eq:Benenti:flat} in the form \eqref{eq:keyiso:anti} shows that
	\name{Benenti} tensors are represented by the algebraic curvature tensors
	\[
		\tfrac12(Ag)\varowedge(Ag)
	\]
	respectively
	\[
		(Au\otimes Au)\varowedge Ag
	\]
	if $M$ is flat.  Here $Ag$ denotes the image of $g$ under the action of $A\in\GLG(V)$ on symmetric tensors on $V$.  We can
	interpret this by saying that \name{Benenti} tensors form the orbit of the metric under the natural action of $\GLG(V)$ on
	\name{Killing} tensors.
\end{example}

\section{The algebraic integrability conditions}
\label{sec:K}

We saw that \name{Killing} tensors on a constant sectional curvature manifold correspond to algebraic curvature tensors.  The
aim of this section is to translate the \name{Nijenhuis} integrability conditions for such \name{Killing} tensors into algebraic
integrability conditions on the corresponding algebraic curvature tensors.

First note that in the integrability conditions \eqref{eq:TNS} the \name{Nijenhuis} torsion \eqref{eq:Nijenhuis} appears only
antisymmetrised in its two lower indices $\beta$,$\gamma$.  To simplify computations we will thus replace the \name{Nijenhuis}
torsion $N\indices{^\alpha_{\beta\gamma}}$ in the integrability conditions by the tensor
\begin{align*}
	\bar N\indices{^\alpha_{\beta\gamma}}&\coloneq
	\tfrac12\bigl(
		K\indices{^\alpha_\delta}\nabla_\gamma K\indices{^\delta_\beta }+
		K\indices{^\delta_\beta }\nabla_\delta K\indices{^\alpha_\gamma}
	\bigr)&
	\bar N\indices{^\alpha_{[\beta\gamma]}}
	&=   N\indices{^\alpha_{ \beta\gamma }}
	\fullstop
\end{align*}
Together with \eqref{eq:KK'} this can be written as
\begin{align*}
	\bar N\indices{^\alpha_{\beta\gamma}}
	=&\;
	S_{a_1a_2b_1b_2}S_{c_1c_2d_1d_2}
	x^{a_1}x^{a_2}x^{c_1}
	\nabla^\alpha x^{b_1}\nabla_\delta x^{b_2}
	\nabla_\gamma x^{c_2}\nabla^\delta x^{d_1}\nabla_\beta x^{d_2}
	\\+&\;
	S_{a_1a_2b_1b_2}S_{c_1c_2d_1d_2}
	x^{a_1}x^{a_2}x^{c_1}
	\nabla^\delta x^{b_1}\nabla_\beta x^{b_2}
	\nabla_\delta x^{c_2}\nabla^\alpha x^{d_1}\nabla_\gamma x^{d_2}
	\fullstop
\end{align*}
\begin{lemma}
	\label{lemma:contraction}
	Let $M$ be one of the standard models for constant sectional curvature manifolds as in \linkref{example}{ex:models}.  Then
	\[
		\nabla_\delta x^a\nabla^\delta x^b=
		\begin{cases}
			g^{ab}-u^au^b & \text{if $M$ is flat}\\
			g^{ab}-x^ax^b & \text{otherwise}
			\fullstop
		\end{cases}
	\]
\end{lemma}
\begin{proof}
	Let $e_1,\ldots,e_n$ be a basis of $T_xM$ and complete it with a unit normal vector $e_0\coloneq u$ to a basis of $V$.
	Then on one hand
	\begin{align*}
		\sum_{i,j=0}^ng(e^i,e^j)\nabla_{e_i}x^a\nabla_{e_j}x^b
		&=\sum_{i,j=1}^ng(e^i,e^j)\nabla_{e_i}x^a\nabla_{e_j}x^b+\nabla_ux^a\nabla_ux^b
		\\
		&=g^{\alpha\beta}\nabla_\alpha x^a\nabla_\beta x^b+u^au^b
		\fullstop
	\end{align*}
	On the other hand, choosing the standard basis of $V$ instead, the left hand side is just $g^{ab}$.  This proves the lemma,
	remarking that we can choose $u=x$ if $M$ is not flat.
\end{proof}
For flat $M$ the lemma yields
\begin{equation}
	\label{eq:N:flat}
	\begin{split}	
		\bar N\indices{^\alpha_{\beta\gamma}}
		=&\;
		\bar g^{b_2d_1}S_{a_1a_2b_1b_2}S_{c_1c_2d_1d_2}
		x^{a_1}x^{a_2}x^{c_1}
		\nabla^\alpha x^{b_1}\nabla_\beta x^{d_2}\nabla_\gamma x^{c_2}
		\phantom{\comma}
		\\+&\;
		\bar g^{b_1c_2}S_{a_1a_2b_1b_2}S_{c_1c_2d_1d_2}
		x^{a_1}x^{a_2}x^{c_1}
		\nabla^\alpha x^{d_1}\nabla_\beta x^{b_2}\nabla_\gamma x^{d_2}
		\comma
	\end{split}
\end{equation}
where $\bar g\coloneq g^{ab}-u^au^b$.  In all other cases we have
\[
	\begin{split}
		\bar N\indices{^\alpha_{\beta\gamma}}
		=&\;
		\bigl(g^{b_2d_1}-x^{b_2}x^{d_1}\bigr)S_{a_1a_2b_1b_2}S_{c_1c_2d_1d_2}
		x^{a_1}x^{a_2}x^{c_1}
		\nabla^\alpha x^{b_1}\nabla_\beta x^{d_2}\nabla_\gamma x^{c_2}
		\phantom{\fullstop}
		\\+&\;
		\bigl(g^{b_1c_2}-x^{b_1}x^{c_2}\bigr)S_{a_1a_2b_1b_2}S_{c_1c_2d_1d_2}
		x^{a_1}x^{a_2}x^{c_1}
		\nabla^\alpha x^{d_1}\nabla_\beta x^{b_2}\nabla_\gamma x^{d_2}
		\fullstop
	\end{split}
\]
But here the two subtracted terms vanish by the \name{Bianchi} identity because they contain the terms
$x^{a_1}x^{a_2}x^{b_2}S_{a_1a_2b_1b_2}$ respectively $x^{a_1}x^{a_2}x^{b_1}S_{a_1a_2b_1b_2}$.  This allows us to use
\eqref{eq:N:flat} for \emph{all} models $M$ if we define
\begin{equation}
	\label{eq:gbar}
	\bar g^{ab}\coloneq
	\begin{cases}
		g^{ab}-u^au^b & \text{if $M$ is flat} \\
		g^{ab}        & \text{otherwise}
		\fullstop
	\end{cases}
\end{equation}
In the case of a hyperplane $M\subset V$, the tensor $\bar g^{ab}$ is the pullback of the metric on $M$ via the orthogonal
projection $V\to M$ and thus degenerated.  Note that we still lower and rise indices with the metric $g^{ab}$ and not with $\bar
g^{ab}$.

In \eqref{eq:N:flat} the lower indices $b_2,d_1$ respectively $b_1,c_2$ are contracted with $\bar g$.  We can make use of the
symmetries of $S_{a_1a_2b_1b_2}$ to bring these indices to the first position:
\[
	\begin{split}
		\bar N\indices{^\alpha_{\beta\gamma}}
		&=
		\bar g^{b_2d_1}S_{b_2b_1a_1a_2}S_{d_1d_2c_1c_2}
		x^{a_1}x^{a_2}x^{c_1}
		\nabla^\alpha x^{b_1}\nabla_\beta x^{d_2}\nabla_\gamma x^{c_2}
		\\&+
		\bar g^{b_1c_2}S_{b_1b_2a_1a_2}S_{c_2c_1d_1d_2}
		x^{a_1}x^{a_2}x^{c_1}
		\nabla^\alpha x^{d_1}\nabla_\beta x^{b_2}\nabla_\gamma x^{d_2}
		\fullstop
	\end{split}
\]
Renaming, lowering and rising indices appropriately finally results in
\begin{equation}
	\label{eq:Nijenhuis:nonanti}
	\bar N_{\alpha\beta\gamma}
	=
	\bar g_{ij}\bigl(
		S\indices{^i_{a_2b_1b_2}}S\indices{^j_{c_2d_1d_2}}+
		S\indices{^i_{c_2b_1b_2}}S\indices{^j_{d_1a_2d_2}}
	\bigr)
	x^{b_1}x^{b_2}x^{d_1}\nabla_\alpha x^{a_2}\nabla_\beta x^{c_2}\nabla_\gamma x^{d_2}
	\fullstop
\end{equation}
In what follows we will substitute this expression together with \eqref{eq:K} into each of the three integrability conditions
\eqref{eq:TNS} and transform them into purely algebraic integrability conditions.

\subsection{The first integrability condition}

The first integrability condition \eqref{eq:TNS:1} can be written as $\bar N_{[\alpha\beta\gamma]}=0$.  For
the expression \eqref{eq:Nijenhuis:nonanti} this is equivalent to the vanishing of the antisymmetrisation in the upper indices
$a_2,c_2,d_2$:
\[
	\bar g_{ij}\bigl(
		S\indices{^i_{a_2b_1b_2}}S\indices{^j_{c_2d_1d_2}}+
		S\indices{^i_{c_2b_1b_2}}S\indices{^j_{d_1a_2d_2}}
	\bigr)
	x^{b_1}x^{b_2}x^{d_1}\nabla_\alpha x^{[a_2}\nabla_\beta x^{c_2}\nabla_\gamma x^{d_2]}
	=0
	\fullstop
\]
Due to the symmetry \eqref{eq:S:sym} the second term to vanishes.  If we write $u$, $v$ and $w$ for the tangent vectors
$\partial_\alpha$, $\partial_\beta$ respectively $\partial_\gamma$ and use \eqref{eq:nabla} in order to get rid of the indices
and $\nabla$'s, we get the condition
\begin{equation}
	\label{eq:paired}
	\begin{split}
		\bar g_{ij}S\indices{^i_{a_2b_1b_2}}S\indices{^j_{c_2d_1d_2}}
		x^{b_1}x^{b_2}x^{d_1}u^{[a_2}v^{c_2}w^{d_2]}
		=0&\\
		\qquad
		\forall x\in M\quad\forall u,v,w\in T_xM&
	\end{split}
\end{equation}
on the symmetrised algebraic curvature tensor $S_{a_1a_2b_1b_2}$.

Note that tensors of the form $x^{b_1}x^{b_2}x^{d_1}u^{[a_2}v^{c_2}w^{d_2]}$ are completely symmetric in the indices
$b_1,b_2,d_1$ and completely antisymmetric in the remaining indices $a_2,c_2,d_2$.  On the level of isomorphism classes the
decomposition of the corresponding $\GLG(V)$-representation space $\Sym^3V\otimes\Lambda^3V$ into irreducible components results
from \eqref{eq:111x3}.  The following lemma gives an explicit realisation of this decomposition in terms of orthogonal
projection operators:
\begin{lemma}
	\label{lem:hooks}
	\begin{align}
		\label{eq:hooks}
		&\frac1{q!}\young(\aone,\smallvdots,\asq)\cdot\frac1{p!}\young(\sone\smallhdots\ssp)\\\notag
		&\qquad=
		\frac{\frac p{q+1}}{(p+q)p!^2q!^2}{\young(\sone\smallhdots\ssp,\aone,\smallvdots,\asq)}         {\young(\sone\smallhdots\ssp,\aone,\smallvdots,\asq)}^\adjoint+
		\frac{\frac q{p+1}}{(p+q)p!^2q!^2}{\young(\aone\sone\smallhdots\ssp ,\smallvdots,\asq)}^\adjoint{\young(\aone\sone\smallhdots\ssp ,\smallvdots,\asq)}
	\end{align}
	In particular, for $p=q=3$:
	\begin{equation}
		\label{eq:projectors}
		\frac1{3!}\young(\ctwo,\dtwo,\atwo)\cdot\frac1{3!}\young(\btwo\bone\done)
		=
		\frac1{2^73^4}{\young(     \btwo\bone\done,\ctwo,\dtwo,\atwo)}         {\young(     \btwo\bone\done,\ctwo,\dtwo,\atwo)}^\adjoint+
		\frac1{2^73^4}{\young(\ctwo\btwo\bone\done      ,\dtwo,\atwo)}^\adjoint{\young(\ctwo\btwo\bone\done      ,\dtwo,\atwo)}
		\fullstop
	\end{equation}
\end{lemma}
\begin{proof}
	Write \eqref{eq:hooks} as $P=P_1+P_2$.  Decomposing temporarily the \name{Young} symmetrisers on the right hand side as in
	\eqref{eq:Young_symmetriser} into a product of a symmetriser and an antisymmetriser and using \eqref{eq:square}, one easily
	checks that $P$, $P_1$ and $P_2$ are orthogonal projectors verifying $P_1P_2=0=P_2P_1$, $PP_1=P_1$ and $PP_2=P_2$.
	Therefore $P_1+P_2$ is an orthogonal projector with image $\im P_1\oplus\im P_2\subseteq\im P$.  The decomposition of the
	isomorphism class of $\im P$ into irreducible components is given by the \name{Littlewood}-\name{Richardson} rule as
	\[
		q\left\{{\young(~,\smallvdots,~)}\right.\boxtimes\overbrace{\young(~\smallhdots~)}^p
		\isom
		\overbrace{\young(~\smallhdots~,~,\smallvdots,~)}^p
		\oplus
		\left.\young(~~\smallhdots~,\smallvdots,~)\right\}q
		\fullstop
	\]
	The \name{Young} frames on the right hand side are those appearing in the expression for $P_1$ respectively $P_2$.  Hence
	they describe the isomorphism classes of $\im P_1$ and $\im P_2$.  This shows that $\im P$ and $\im (P_1+P_2)=\im
	P_1\oplus\im P_2$ have the same dimension and are thus equal.  This implies $P=P_1+P_2$.
\end{proof}
\begin{remark}
	The lemma can be interpreted as an explicit splitting of the terms in the long exact sequence
	\[
		0
		\rightarrow\Sym^dV
		\rightarrow\ldots
		\rightarrow\Sym^pV\otimes\Lambda^qV
		\rightarrow\Sym^{p-1}\otimes\Lambda^{q+1}
		\rightarrow\ldots
		\rightarrow\Lambda^dV
		\rightarrow0
		\comma
	\]
	known as the \dfn{\name{Koszul} complex}.
\end{remark}
Applying \eqref{eq:projectors} to the tensor $x^{b_1}x^{b_2}x^{d_1}u^{[a_2}v^{c_2}w^{d_2]}$, we conclude that
\[
	\begin{split}
		&x^{b_1}x^{b_2}x^{d_1}u^{[a_2}v^{c_2}w^{d_2]}\\
		&\quad
		=\frac1{2^73^4}
		\left(
			 \young(\btwo \bone \done){\young(\btwo,\ctwo,\dtwo,\atwo)}^2\young(\btwo \bone \done)
			+\young(\ctwo,\dtwo,\atwo){\young(\ctwo \btwo \bone \done)}^2\young(\ctwo,\dtwo,\atwo)
		\right)
		x^{b_1}x^{b_2}x^{d_1}u^{[a_2}v^{c_2}w^{d_2]}\\&\quad
		=\text{constant}\cdot
		\left(
			 {\young(\btwo\bone\done,\ctwo,\dtwo,\atwo)}
			+{\young(\ctwo\btwo\bone \done,\dtwo,\atwo)}^\adjoint
		\right)
		x^{b_1}x^{b_2}x^{d_1}u^{a_2}v^{c_2}w^{d_2}
		\fullstop
	\end{split}
\]
In the last step we omitted an explicit antisymmetrisation in $a_2$, $c_2$, $d_2$, since this is already carried out implicitely
by each of the \name{Young} symmetrisers.  Accordingly the left hand side of \eqref{eq:paired} splits into two terms.  The
following lemma shows that the second of them, namely
\begin{align*}
	&
    \bar g_{ij}S\indices{^i_{a_2b_1b_2}}S\indices{^j_{c_2d_1d_2}}
    \left(
		{\young(\ctwo\btwo\bone\done,\dtwo,\atwo)}^\adjoint
		x^{b_1}x^{b_2}x^{d_1}u^{a_2}v^{c_2}w^{d_2}
	\right)
	\\&\qquad=
    \left(
		\young(\ctwo\btwo\bone\done,\dtwo,\atwo)
		\bar g_{ij}S\indices{^i_{a_2b_1b_2}}S\indices{^j_{c_2d_1d_2}}
	\right)
	x^{b_1}x^{b_2}x^{d_1}u^{a_2}v^{c_2}w^{d_2}
	\comma
\end{align*}
vanishes identically.
\begin{lemma}
	\begin{equation}
		\label{eq:hook=0}
		\young(\ctwo\btwo\bone\done,\dtwo,\atwo)
		\bar g_{ij}S\indices{^i_{a_2b_1b_2}}S\indices{^j_{c_2d_1d_2}}
		=0
	\end{equation}
\end{lemma}
Before we prove the lemma, we mention an identity which we will frequently use and which is obtained from symmetrising the
\name{Bianchi} identity \eqref{eq:S:Bianchi} in $b_1,b_2$:
\begin{equation}
	\label{eq:Bianchi:sym}
	\young(\bone\btwo)S\indices{^i_{a_2b_1b_2}}=-2\young(\bone\btwo)S\indices{^i_{b_1b_2a_2}}
	\fullstop
\end{equation}
We refer to this identity as \dfn{symmetrised \name{Bianchi} identity}.
\begin{proof}
	\[
		\begin{array}{r@{}l}
			\multicolumn{2}{l}{
				\young(\ctwo\btwo\bone\done,\dtwo,\atwo)
				\bar g_{ij}S\indices{^i_{a_2b_1b_2}}S\indices{^j_{c_2d_1d_2}}
				=
				\young(\ctwo\btwo\bone\done)\young(\ctwo,\dtwo,\atwo)
				\bar g_{ij}S\indices{^i_{a_2b_1b_2}}S\indices{^j_{c_2d_1d_2}}
			}
			\\[2\bigskipamount]
			\qquad=\young(\ctwo\btwo\bone\done)g_{ij}
			\bigl(
				\phantom
				+S\indices{^i_{a_2b_1b_2}}S\indices{^j_{c_2d_1d_2}}&-S\indices{^i_{d_2b_1b_2}}S\indices{^j_{c_2d_1a_2}}\\
				+S\indices{^i_{d_2b_1b_2}}S\indices{^j_{a_2d_1c_2}}&-S\indices{^i_{a_2b_1b_2}}S\indices{^j_{d_2d_1c_2}}\\
				+S\indices{^i_{c_2b_1b_2}}S\indices{^j_{d_2d_1a_2}}&-S\indices{^i_{c_2b_1b_2}}S\indices{^j_{a_2d_1d_2}}
			\bigr)
			\end{array}
	\]
	Regard the parenthesis under complete symmetrisation in $c_2,b_2,b_1,d_1$.  The last two terms vanish due to the
	\name{Bianchi} identity \eqref{eq:Bianchi:alt}.  Renaming $i,j$ as $j,i$ in the third term shows that it cancels the fourth.
	That the first two also cancel each other can be seen by applying twice the symmetrised \name{Bianchi} identity
	\eqref{eq:Bianchi:sym}, once to $S\indices{^i_{a_2b_1b_2}}$ and once to $S\indices{^i_{c_2d_1d_2}}$.
\end{proof}
\begin{remark}
	If the inner product $g$ is positive definite, then $M$ is the unit sphere.  In this case the lemma above also follows from
	the symmetry classification of \name{Riemann} tensor polynomials \cite{Fulling&King&Wybourne&Cummins}, since the tensor
	$g_{ij}R\indices{^i_{a_2b_1b_2}}R\indices{^j_{c_2d_1d_2}}$ has symmetry type
	\[
		\yng(4,2)\oplus\yng(3,2,1)\oplus\yng(3,1,1,1)\oplus\yng(2,2,2)
	\]
	and $g_{ij}S\indices{^i_{a_2b_1b_2}}S\indices{^j_{c_2d_1d_2}}$ can be expressed in terms of this tensor via
	\eqref{eq:R->S}.
\end{remark}
Resuming, the first integrability condition is equivalent to
\begin{equation}
	\label{eq:restricted}
	\begin{split}
        \bar g_{ij}S\indices{^i_{a_2b_1b_2}}S\indices{^j_{c_2d_1d_2}}
		\left(
			\young(\btwo\bone\done,\ctwo,\dtwo,\atwo)
			x^{b_1}x^{b_2}x^{d_1}u^{a_2}v^{c_2}w^{d_2}
		\right)
		=0&\\
		\forall x\in M\quad\forall u,v,w\in T_xM&
		\fullstop
	\end{split}
\end{equation}
We can drop the restriction $u,v,w\in T_xM$ in \eqref{eq:restricted}.  Indeed, decomposing $u,v,w\in V$ according to the
splitting $V=T_xM\oplus\R x$ shows
\[
	\young(\btwo\bone\done,\ctwo,\dtwo,\atwo)
	x^{b_1}x^{b_2}x^{d_1}u^{a_2}v^{c_2}w^{d_2}
	=0
	\qquad
	\text{if $u=x$ or $v=x$ or $w=x$}
	\fullstop
\]
This follows from \name{Dirichlet}'s drawer principle.  This trick is crucial, as it allows us to deal with
$\GLG(N)$-representations instead of the more complicated $\OG(N)$-representations.  Obviously, we can also drop the restriction
$x\in M$ since $\R M$ is dense in $V$.

Next we use the fact that tensors of the form $x^{b_1}x^{b_2}x^{d_1}$ and tensors of the form
$\young(\btwo\bone\done)x^{b_1}y^{b_2}z^{d_1}$ both span the same space, namely $\Sym^3V$.  With this remark condition
\eqref{eq:restricted} is now equivalent to
\[
	\begin{split}
        \bar g_{ij}S\indices{^i_{a_2b_1b_2}}S\indices{^j_{c_2d_1d_2}}
		\young(\btwo\bone\done,\ctwo,\dtwo,\atwo)
		\young(\btwo\bone\done)
		x^{b_1}y^{b_2}z^{d_1}u^{a_2}v^{c_2}w^{d_2}
		=0&\\
		\forall x,y,z,u,v,w\in V&
		\fullstop
	\end{split}
\]
But the operator
\begin{align*}
	&
	\young(\btwo\bone\done,\ctwo,\dtwo,\atwo)
	\young(\btwo\bone\done)
	=
	\young(\btwo\bone\done)
	\young(\btwo,\ctwo,\dtwo,\atwo)
	\young(\btwo\bone\done)
	=
	\frac1{4!}
	\young(\btwo\bone\done)
	\young(\btwo,\ctwo,\dtwo,\atwo)
	\young(\btwo,\ctwo,\dtwo,\atwo)
	\young(\btwo\bone\done)
	\\&\quad=
	\frac1{4!}
	{\young(\btwo\bone\done,\ctwo,\dtwo,\atwo)}
	{\young(\btwo\bone\done,\ctwo,\dtwo,\atwo)}^\adjoint
\end{align*}
is self-adjoint and hence
\begin{multline*}
	\bar g_{ij}S\indices{^i_{a_2b_1b_2}}S\indices{^j_{c_2d_1d_2}}
	\left(
		\young(\btwo\bone\done,\ctwo,\dtwo,\atwo)
		\young(\btwo\bone\done)
		x^{b_1}y^{b_2}z^{d_1}u^{a_2}v^{c_2}w^{d_2}
	\right)
	\\=
	\frac1{4!}
	\left(
		{\young(\btwo\bone\done,\ctwo,\dtwo,\atwo)}
		{\young(\btwo\bone\done,\ctwo,\dtwo,\atwo)}^\adjoint
        \bar g_{ij}S\indices{^i_{a_2b_1b_2}}S\indices{^j_{c_2d_1d_2}}
	\right)
	x^{b_1}y^{b_2}z^{d_1}u^{a_2}v^{c_2}w^{d_2}
	\fullstop
\end{multline*}
Now recall that $V^{\otimes6}$ is spanned by tensors of the form $x^{b_1}y^{b_2}z^{d_1}u^{a_2}v^{c_2}w^{d_2}$ and that the dual
pairing $\bar V^{\otimes6}\times V^{\otimes 6}\to\R$ is non-degenerate.  This allows us finally to write the first integrability
condition in the purely algebraic form
\begin{equation}
	\label{eq:penultimate}
	{\young(\btwo\bone\done,\ctwo,\dtwo,\atwo)}
	{\young(\btwo\bone\done,\ctwo,\dtwo,\atwo)}^\adjoint
	\bar g_{ij}S\indices{^i_{a_2b_1b_2}}S\indices{^j_{c_2d_1d_2}}=0
	\comma
\end{equation}
which is independent of $x,y,z,u,v,w$.  We will now give a number of equivalent formulations.

\begin{proposition}[First integrability condition]
	The following conditions are equivalent to the first integrability condition \eqref{eq:TNS:1} for a \name{Killing} tensor on
	a constant sectional curvature manifold $M$:
	\begin{enumerate}
		\item
			\label{it:KS}
			The corresponding symmetrised algebraic curvature tensor $S$ satisfies
			\begin{equation}
				\label{eq:KS}
				P\;\bar g_{ij}S\indices{^i_{a_2b_1b_2}}S\indices{^j_{c_2d_1d_2}}=0
				\phantomequation{eq:KS:a}{a}
				\phantomequation{eq:KS:b}{b}
				\phantomequation{eq:KS:c}{c}
				\phantomequation{eq:KS:d}{d}
				\comma
			\end{equation}
			where $P$ is any of the following symmetry operators
			\footnote{Up to a constant they are all projectors.}
			\begin{align}
				\label{eq:operators}
				(a)\;&{\young(\btwo\bone\done,\ctwo,\dtwo,\atwo)}^\adjoint&
				(b)\;&{\young(\ctwo,\dtwo,\atwo)                }
				      {\young(\btwo\bone\done)                  }&
				(c)\;&{\young(\btwo,\ctwo,\dtwo,\atwo)          }&
				(d)\;&{\young(\btwo\bone\done,\ctwo,\dtwo,\atwo)}
			\end{align}
		\item
			\label{it:KR}
			The corresponding algebraic curvature tensor $R$ satisfies
			\begin{equation}
				\label{eq:KR}
				P\;\bar g_{ij}R\indices{^i_{b_1a_2b_2}}R\indices{^j_{d_1c_2d_2}}=0
				\phantomequation{eq:KR:c}{c}
				\comma
			\end{equation}
			where $P$ is any of the symmetry operators \eqref{eq:operators}.
	\end{enumerate}
	If $M$ is not flat, this is is in addition equivalent to:
	\begin{enumerate}[resume]
		\item
			The curvature form $\Omega\in\End(V)\otimes\Lambda^2V$of $R$, defined by
			\[
				\Omega\indices{^{a_1}_{b_1}}\coloneq R\indices{^{a_1}_{b_1a_2b_2}}dx^{a_2}\wedge dx^{b_2}
			\]
			satisfies
			\begin{equation}
				\label{eq:Omega}
				\Omega\wedge\Omega=0
				\comma
			\end{equation}
			where the wedge product is defined by taking the exterior product in the $\Lambda^2V$-component and usual matrix
			multiplication in the $\End(V)$-component.
	\end{enumerate}
\end{proposition}
\begin{remark}
	In \eqref{eq:operators} we can permute the labels $a_2,b_2,c_2,d_2$ arbitrarily as well as exchange the labels $b_1,d_1$.
	This follows from the integrability condition in the form \eqref{eq:KS:c}.  In particular, in \eqref{eq:KS:b} one can
	antisymmetrise in any three of the four indices $a_2,b_2,c_2,d_2$ and symmetrise in the remaining three.
\end{remark}
\begin{proof}
	We showed that the first integrability condition \eqref{eq:TNS:1} is equivalent to \eqref{eq:penultimate}.  But this is
	equivalent to condition \eqref{eq:KS:a} since the kernels of $PP^\adjoint$ and $P^\adjoint$ coincide:
	\begin{equation}
		\label{eq:kernels}
		\quad PP^\adjoint v=0
		\quad\Leftrightarrow\quad\lVert P^\adjoint v\rVert^2=\langle v|PP^\adjoint v\rangle=0
		\quad\Leftrightarrow\quad P^\adjoint v=0
		\fullstop
	\end{equation}
	The equivalence \eqref{eq:KS:a} $\Leftrightarrow$ \eqref{eq:KS:b} follows from \eqref{eq:projectors} combined with
	\eqref{eq:hook=0}.

	The implication \eqref{eq:KS:c} $\Rightarrow$ \eqref{eq:KS:d} is trivial.  We finish the proof of part \ref{it:KS} by
	proving \eqref{eq:KS:d} $\Rightarrow$ \eqref{eq:KS:b} $\Rightarrow$ \eqref{eq:KS:c} through a stepwise manipulation of
	\begin{subequations}
		\begin{equation}
			\label{eq:one}
			\young(\btwo\bone\done,\ctwo,\dtwo,\atwo)
			\bar g_{ij}S\indices{^i_{a_2b_1b_2}}S\indices{^j_{c_2d_1d_2}}
			=
			\young(\btwo\bone\done)
			\young(\btwo,\ctwo,\dtwo,\atwo)
			\bar g_{ij}S\indices{^i_{a_2b_1b_2}}S\indices{^j_{c_2d_1d_2}}
			\fullstop
		\end{equation}
		In order to sum over all $q!$ permutations of $q$ indices, one can take the sum over $q$ cyclic permutations, chose
		one index and then sum over all $(q-1)!$ permutations of the remaining $(q-1)$ indices.  Apply this to the
		antisymmetrisation in $a_2,b_2,c_2,d_2$ (fixing $b_2$):
		\begin{equation}
			\label{eq:two}
			\begin{split}
				\eqref{eq:one}
				=
				\young(\btwo\bone\done)
				\young(\ctwo,\dtwo,\atwo)
				\bar g_{ij}
				\Bigl(
					&S\indices{^i_{\underline a_2b_1           b_2}}S\indices{^j_{\underline c_2d_1\underline d_2}}-
					 S\indices{^i_{           b_2b_1\underline c_2}}S\indices{^j_{\underline d_2d_1\underline a_2}}\\+\,
					&S\indices{^i_{\underline c_2b_1\underline d_2}}S\indices{^j_{\underline a_2d_1           b_2}}-
					 S\indices{^i_{\underline d_2b_1\underline a_2}}S\indices{^j_{           b_2d_1\underline c_2}}
				\Bigr)
				\fullstop
			\end{split}
		\end{equation}
		For a better readability we underlined each antisymmetrised index.  Now use the symmetrised \name{Bianchi} identity
		\eqref{eq:Bianchi:sym} to bring the index $c_2$ from the fourth to the second index position:
		\begin{equation}	
			\label{eq:three}
			\begin{split}
				\eqref{eq:two}
				=
				\young(\btwo\bone\done)
				\young(\ctwo,\dtwo,\atwo)
				\bar g_{ij}
				\Bigl(
					\phantom+
					&         S\indices{^i_{\underline a_2b_1           b_2}}S\indices{^j_{\underline c_2d_1\underline d_2}}
					 +\tfrac12S\indices{^i_{\underline c_2b_1           b_2}}S\indices{^j_{\underline d_2d_1\underline a_2}}\\+\,
					&         S\indices{^i_{\underline c_2b_1\underline d_2}}S\indices{^j_{\underline a_2d_1           b_2}}
					 +\tfrac12S\indices{^i_{\underline d_2b_1\underline a_2}}S\indices{^j_{\underline c_2d_1           b_2}}
				\Bigr)
				\fullstop
			\end{split}
		\end{equation}
		Then rename $i,j$ as $j,i$ in the last two terms:
		\begin{equation}
			\label{eq:four}
			\begin{split}
				\eqref{eq:three}
				=
				\young(\btwo\bone\done)
				\young(\ctwo,\dtwo,\atwo)
				\bar g_{ij}
				\Bigl(
					\phantom+
					&         S\indices{^i_{\underline a_2b_1b_2}}S\indices{^j_{\underline c_2d_1\underline d_2}}
					 +\tfrac12S\indices{^i_{\underline c_2b_1b_2}}S\indices{^j_{\underline d_2d_1\underline a_2}}\\+\,
					&         S\indices{^i_{\underline a_2d_1b_2}}S\indices{^j_{\underline c_2b_1\underline d_2}}
					 +\tfrac12S\indices{^i_{\underline c_2d_1b_2}}S\indices{^j_{\underline d_2b_1\underline a_2}}
				\Bigr)
				\fullstop
			\end{split}
		\end{equation}
		Finally use the symmetrisation in $b_2,b_1,d_1$ and the antisymmetrisation in $c_2,d_2,a_2$ to bring each term to the
		same form:
		\begin{equation}
			\label{eq:five}
			\eqref{eq:four}
			=
			\young(\btwo\bone\done)
			\young(\ctwo,\dtwo,\atwo)
			\bar g_{ij}
			\Bigl(
				3S\indices{^i_{\underline a_2b_1b_2}}S\indices{^j_{\underline c_2d_1\underline d_2}}
			\Bigr)
			\fullstop
		\end{equation}
	\end{subequations}
	This proves \eqref{eq:KS:d} $\Leftrightarrow$ \eqref{eq:KS:b}.  To continue, antisymmetrise 
	\begin{align*}
		0=
		\young(\ctwo,\dtwo,\atwo)
		\young(\btwo\bone\done)
		\bar g_{ij}&S\indices{^i_{\underline a_2b_1b_2}}S\indices{^j_{\underline c_2d_1\underline d_2}}\\
		=
		2\young(\ctwo,\dtwo,\atwo)
		\bar g_{ij}
		\Bigl(
			&S\indices{^i_{\underline a_2b_1b_2}}S\indices{^j_{\underline c_2d_1\underline d_2}}+
			 S\indices{^i_{\underline a_2b_2d_1}}S\indices{^j_{\underline c_2b_1\underline d_2}}+
			 S\indices{^i_{\underline a_2d_1b_1}}S\indices{^j_{\underline c_2b_2\underline d_2}}
		\Bigr)
	\end{align*}
	in $a_2,b_2,c_2,d_2$.  Then the last term vanishes by the symmetry \eqref{eq:S:sym}, yielding
	\[
		0=
		\young(\atwo,\btwo,\ctwo,\dtwo)
		\bar g_{ij}
		\Bigl(
			S\indices{^i_{\underline a_2b_1\underline b_2}}S\indices{^j_{\underline c_2d_1\underline d_2}}+
			S\indices{^i_{\underline a_2d_1\underline b_2}}S\indices{^j_{\underline c_2b_1\underline d_2}}
		\Bigr)
		\fullstop
	\]
	Both sum terms are equal under antisymmetrisation in $a_2,b_2,c_2,d_2$ and contraction with $\bar g_{ij}$.  Indeed,
	exchanging $b_1$ and $d_1$ is tantamount to exchanging $a_2$ with $c_2$ and $b_2$ with $d_2$ and renaming $i,j$ as $j,i$.
	This proves \eqref{eq:KS:b} $\Rightarrow$ \eqref{eq:KS:c}.

	From the correspondence \eqref{eq:R<->S} between $R$ and $S$ we conclude the equivalence \eqref{eq:KS:c} $\Leftrightarrow$
	\eqref{eq:KR:c}.  The proof of the remaining part of \ref{it:KR} is completely analogous to the proof of \ref{it:KS}, so we
	leave it to the reader.  Condition \eqref{eq:Omega} is just a reformulation of \eqref{eq:KR:c}.  This finishes the proof.
\end{proof}
\begin{remark}
	\label{rem:shuffle}
	In the preceeding proof we made use of a particular notation as well as some particular tensor index manipulations.  We will
	do this several times in more complex computations during the next two sections, so we would like to make this explicit.
	\begin{itemize}
		\item
			First, we call a \name{Young} symmetriser as in \eqref{eq:one}, which is the product of a symmetriser and an
			antisymmetriser sharing a common label (and thus not commuting) a \dfn{hook symmetriser}.  Note that \eqref{eq:one}
			and \eqref{eq:two} are merely different ways to write down the same term, using a smaller antisymmetriser but
			applied to more terms.  We call this to \dfn{reduce an antisymmetriser by a label} ($b_2$ in this case).  This works
			likewise for a symmetriser and allows us to replace any hook symmetriser by a product of a symmetriser and an
			antisymmetriser with disjoint label sets (and thus both commuting).  The latter are more easy to deal with.  We call
			this procedure \emph{splitting a hook symmetriser}.
		\item
			Second, for better readability we stick to the above notation and underline antisymmetrised tensor indices as in
			\eqref{eq:three}.
		\item
			Third, regard the manipulations from \eqref{eq:three} to \eqref{eq:five}.  What we did is to bring the indices of
			every term in \eqref{eq:three} to the same order as in $g_{ij}S\indices{^i_{a_2b_1b_2}}S\indices{^j_{c_2d_1d_2}}$ by
			using:
			\begin{itemize}
				\item the symmetry in $i,j$ under contraction with $\bar g_{ij}$,
				\item the (anti)symmetry under the (anti)symmetriser and
				\item the symmetries of $S$ itself, especially the symmetrised \name{Bianchi} identity \eqref{eq:Bianchi:sym}.
			\end{itemize}
		We will call this procedure \emph{reordering indices}.
	\end{itemize}
\end{remark}

\subsection{The second integrability condition}

The proceeding for the remaining two integrability conditions is similar as for the first one, only longer.  We therefore treat
both in parallel as far as possible and shorten the explications where they are analogous.  We begin by substituting the
expressions \eqref{eq:Nijenhuis:nonanti} and \eqref{eq:K} into the tensors appearing in \eqref{eq:TNS:2} and \eqref{eq:TNS:3}:
\begin{align*}
	&\bar N\indices{^\delta_{\beta\gamma}}K_{\delta\alpha}\\
	&\quad=
	\bar g_{ij}
	\bigl(
		S\indices{^i_{a_2b_1b_2}}S\indices{^j_{c_2d_1d_2}}+
		S\indices{^i_{c_2b_1b_2}}S\indices{^j_{d_1a_2d_2}}
	\bigr)
	x^{b_1}x^{b_2}x^{d_1}\nabla^\delta x^{a_2}\nabla_\beta x^{c_2}\nabla_\gamma x^{d_2}
	\\&\qquad\qquad
	S_{e_1e_2f_1f_2}
	x^{e_1}x^{e_2}\nabla_\delta x^{f_1}\nabla_\alpha x^{f_2}
	\\
	&\bar N\indices{^\delta_{\beta\gamma}}K_{\varepsilon\alpha}K\indices{^\varepsilon_\delta}\\
	&\quad=
	\bar g_{ij}
	\bigl(
		S\indices{^i_{a_2b_1b_2}}S\indices{^j_{c_2d_1d_2}}+
		S\indices{^i_{c_2b_1b_2}}S\indices{^j_{d_1a_2d_2}}
	\bigr)
	x^{b_1}x^{b_2}x^{d_1}\nabla^\delta x^{a_2}\nabla_\beta x^{c_2}\nabla_\gamma x^{d_2}
	\\&\qquad\qquad
	S_{e_1e_2f_1f_2}
	x^{e_1}x^{e_2}\nabla_\varepsilon x^{f_1}\nabla_\alpha x^{f_2}
	\;
	S_{g_1g_2h_1h_2}
	x^{g_1}x^{g_2}\nabla^\varepsilon x^{h_1}\nabla_\delta x^{h_2}
	\fullstop
\end{align*}
As before, we replace the contractions over $\delta$ and $\varepsilon$ according to \linkref{lemma}{lemma:contraction}
and omit the terms that vanish according to the \name{Bianchi} identity:
\begin{align*}
	&\bar N\indices{^\delta_{\beta\gamma}}K_{\delta\alpha}\\
	&\quad=
	\bar g_{ij}\bar g^{a_2f_1}
	\bigl(
		S\indices{^i_{a_2b_1b_2}}S\indices{^j_{c_2d_1d_2}}+
		S\indices{^i_{c_2b_1b_2}}S\indices{^j_{d_1a_2d_2}}
	\bigr)
	S_{e_1e_2f_1f_2}
	\\&\qquad\qquad
	x^{b_1}x^{b_2}x^{d_1}x^{e_1}x^{e_2}
	\nabla_\beta x^{c_2}\nabla_\gamma x^{d_2}\nabla_\alpha x^{f_2}
	\\
	&\bar N\indices{^\delta_{\beta\gamma}}K_{\varepsilon\alpha}K\indices{^\varepsilon_\delta}\\
	&\quad=
	\bar g_{ij}\bar g^{a_2h_2}\bar g^{f_1h_1}
	\bigl(
		S\indices{^i_{a_2b_1b_2}}S\indices{^j_{c_2d_1d_2}}+
		S\indices{^i_{c_2b_1b_2}}S\indices{^j_{d_1a_2d_2}}
	\bigr)
	S_{e_1e_2f_1f_2}
	S_{g_1g_2h_1h_2}
	\\&\qquad\qquad
	x^{b_1}x^{b_2}x^{d_1}x^{e_1}x^{e_2}x^{g_1}x^{g_2}
	\nabla_\beta x^{c_2}\nabla_\gamma x^{d_2}\nabla_\alpha x^{f_2}
	\fullstop
\end{align*}
The integrability conditions \eqref{eq:TNS:2} and \eqref{eq:TNS:3} are equivalent to the vanishing of the antisymmetrisation of
the above tensors in $\alpha,\beta,\gamma$.  As before, this can be written as
\begin{equation}
	\label{eq:23:paired}
	\begin{array}{r@{}c@{}c@{\;\,}l}
		\bar g_{ij}\bar g_{kl}&
		\multicolumn{1}{@{}l}{
			\bigl(
				S\indices{^{ik}_{b_1b_2}}S\indices{^j_{c_2d_1d_2}}+
				S\indices{^i_{c_2b_1b_2}}S\indices{^j_{d_1}^k_{d_2}}
			\bigr)
			S\indices{^l_{f_2e_1e_2}}
		}
		\\[\medskipamount]&
		\multicolumn{1}{r}{
			x^{b_1}x^{b_2}x^{d_1}x^{e_1}x^{e_2}
			u^{[c_2}v^{d_2}w^{f_2]}
		}
		&=&0
		\\[\medskipamount]
		\bar g_{ij}\bar g_{kl}\bar g_{mn}&
		\multicolumn{1}{@{}l}{
			\bigl(
				S\indices{^{ik}_{b_1b_2}}S\indices{^j_{c_2d_1d_2}}+
				S\indices{^i_{c_2b_1b_2}}S\indices{^j_{d_1}^k_{d_2}}
			\bigr)
			S\indices{^m_{f_2e_1e_2}}
			S\indices{^{nl}_{g_1g_2}}
			\qquad
		}
		\\[\medskipamount]&
		\multicolumn{1}{r}{
			x^{b_1}x^{b_2}x^{d_1}x^{e_1}x^{e_2}x^{g_1}x^{g_2}
			u^{[c_2}v^{d_2}w^{f_2]}
		}
		&=&0
		\\
		\forall x\in M\quad&
		\multicolumn{1}{l}{
			\forall u,v,w\in T_xM
			\fullstop
		}
	\end{array}
\end{equation}
The tensors
\begin{align*}
	x^{b_1}x^{b_2}x^{d_1}x^{e_1}x^{e_2}              &u^{[c_2}v^{d_2}w^{f_2]}&
	x^{b_1}x^{b_2}x^{d_1}x^{e_1}x^{e_2}x^{g_1}x^{g_2}&u^{[c_2}v^{d_2}w^{f_2]}
\end{align*}
are antisymmetric in $c_2,d_2,f_2$ and symmetric in the remaining indices.  We decompose them according to
\linkref{lemma}{lem:hooks}.  This yields
\begin{align*}
	&x^{b_1}x^{b_2}x^{d_1}x^{e_1}x^{e_2}              u^{[c_2}v^{d_2}w^{f_2]}\\
	&\quad       =\text{constant}\cdot{\young(     \btwo\bone\done\eone\etwo          ,\ctwo,\dtwo,\ftwo)}          x^{b_1}x^{b_2}x^{d_1}x^{e_1}x^{e_2}              u^{c_2}v^{d_2}w^{f_2}\\
	&\qquad\qquad+\text{constant}\cdot{\young(\ctwo\btwo\bone\done\eone\etwo                ,\dtwo,\ftwo)}^\adjoint x^{b_1}x^{b_2}x^{d_1}x^{e_1}x^{e_2}              u^{c_2}v^{d_2}w^{f_2}
	\intertext{and}
	&x^{b_1}x^{b_2}x^{d_1}x^{e_1}x^{e_2}x^{g_1}x^{g_2}u^{[c_2}v^{d_2}w^{f_2]}\\
	&\quad       =\text{constant}\cdot{\young(     \btwo\bone\done\eone\etwo\gone\gtwo,\ctwo,\dtwo,\ftwo)}          x^{b_1}x^{b_2}x^{d_1}x^{e_1}x^{e_2}x^{g_1}x^{g_2}u^{c_2}v^{d_2}w^{f_2}\\
	&\qquad\qquad+\text{constant}\cdot{\young(\ctwo\btwo\bone\done\eone\etwo\gone\gtwo      ,\dtwo,\ftwo)}^\adjoint x^{b_1}x^{b_2}x^{d_1}x^{e_1}x^{e_2}x^{g_1}x^{g_2}u^{c_2}v^{d_2}w^{f_2}
	\fullstop
\end{align*}
The following lemma shows that, when substituted into \eqref{eq:23:paired}, only the first term is relevant in each case:
\begin{lemma}
	\label{lem:hook=0:23}
	\begin{subequations}
		\label{eq:hooks=0}
		\begin{alignat}{4}
			\label{eq:hook=0:2:yin}
			\young(\ctwo\btwo\bone\done\eone\etwo,\dtwo,\ftwo)
			&\bar g_{ij}\bar g_{kl}
			&S\indices{^{ik}_{b_1b_2}}
			&S\indices{^j_{c_2d_1d_2}}
			&S\indices{^l_{f_2e_1e_2}}
			&=0\\
			\label{eq:hook=0:2:yang}
			\young(\ctwo\btwo\bone\done\eone\etwo,\dtwo,\ftwo)
			&\bar g_{ij}\bar g_{kl}
			&S\indices{^i_{c_2b_1b_2}}
			&S\indices{^j_{d_1}^k_{d_2}}
			&S\indices{^l_{f_2e_1e_2}}
			&=0
		\end{alignat}
	\end{subequations}
	\begin{subequations}
		\label{eq:hook=0:3}
		\begin{alignat}{4}
			\label{eq:hook=0:3:yin}
			\young(\ctwo\btwo\bone\done\eone\etwo\gone\gtwo,\dtwo,\ftwo)
			&\bar g_{ij}\bar g_{kl}\bar g_{mn}
			&S\indices{^{ik}_{b_1b_2}}
			&S\indices{^j_{c_2d_1d_2}}
			&S\indices{^m_{f_2e_1e_2}}
			&S\indices{^{nl}_{g_1g_2}}
			&=0\\
			\label{eq:hook=0:3:yang}
			\young(\ctwo\btwo\bone\done\eone\etwo\gone\gtwo,\dtwo,\ftwo)
			&\bar g_{ij}\bar g_{kl}\bar g_{mn}
			&S\indices{^i_{c_2b_1b_2}}
			&S\indices{^j_{d_1}^k_{d_2}}
			&S\indices{^m_{f_2e_1e_2}}
			&S\indices{^{nl}_{g_1g_2}}
			&=0
		\end{alignat}
	\end{subequations}
\end{lemma}
\begin{proof}
	Expanding the antisymmetriser of the \name{Young} symmetriser on the left hand side of \eqref{eq:hook=0:2:yin} yields
	\begin{alignat*}{4}
		\young(\ctwo\btwo\bone\done\eone\etwo)
		\bar g_{ij}\bar g_{kl}S\indices{^{ik}_{b_1b_2}}
		\bigl(
			& &S\indices{^j_{c_2d_1d_2}}&S\indices{^l_{f_2e_1e_2}}&-&S\indices{^j_{c_2d_1f_2}}&S\indices{^l_{d_2e_1e_2}}&\\
			&+&S\indices{^j_{f_2d_1c_2}}&S\indices{^l_{d_2e_1e_2}}&-&S\indices{^j_{d_2d_1c_2}}&S\indices{^l_{f_2e_1e_2}}&\\
			&+&S\indices{^j_{d_2d_1f_2}}&S\indices{^l_{c_2e_1e_2}}&-&S\indices{^j_{f_2d_1d_2}}&S\indices{^l_{c_2e_1e_2}}&
		\bigr)
		\fullstop
	\end{alignat*}
	Now regard the parenthesis under complete symmetrisation in $b_1,b_2,c_2,d_1,e_1,e_2$.  The last two terms vanish by the
	\name{Bianchi} identity.  Renaming $i,j,k,l$ as $k,l,i,j$ in the third term shows that it cancels the fourth due to the
	contraction with  $\bar g_{ij}\bar g_{kl}S\indices{^{ik}_{b_1b_2}}$.  That the first two also cancel each other can be seen
	after applying twice the symmetrised \name{Bianchi} identity, once to $S\indices{^j_{c_2d_1d_2}}$ and once to
	$S\indices{^l_{f_2e_1e_2}}$.

	In the same way, the left hand side of \eqref{eq:hook=0:2:yang}, written without terms vanishing by the \name{Bianchi}
	identity, is
	\[
		\young(\ctwo\btwo\bone\done\eone\etwo)
		\bar g_{ij}\bar g_{kl}S\indices{^j_{d_1}^k_{c_2}}
		\bigl(
			 S\indices{^i_{d_2b_1b_2}}S\indices{^l_{f_2e_1e_2}}
			-S\indices{^i_{f_2b_1b_2}}S\indices{^l_{d_2e_1e_2}}
		\bigr)
		\fullstop
	\]
	Renaming $i,j,k,l$ as $l,k,j,i$ in the first term shows that this is zero too.  The proof of \eqref{eq:hook=0:3} is
	straightforward, using the same arguments.  We leave this to the reader.
\end{proof}
\begin{remark}
	For the unit sphere, the lemma also follows from the symmetry classification of \name{Riemann} tensor polynomials
	\cite{Fulling&King&Wybourne&Cummins}.  Indeed, the tensors under the \name{Young} symmetriser in \eqref{eq:hooks=0} and
	\eqref{eq:hook=0:3} can be expressed in terms of the corresponding algebraic curvature tensor via \eqref{eq:R->S} and the
	resulting tensors have no $\textYF\yng(5,1,1)$ respectively $\textYF\yng(7,1,1)$ component.	
\end{remark}
We have shown the equivalence of the second and third integrability condition to
\[
	\begin{array}{r@{}c@{}c@{\;\,}l}
		\bar g_{ij}\bar g_{kl}&
		\multicolumn{1}{@{}l}{
			\bigl(
				S\indices{^{ik}_{b_1b_2}}S\indices{^j_{c_2d_1d_2}}+
				S\indices{^i_{c_2b_1b_2}}S\indices{^j_{d_1}^k_{d_2}}
			\bigr)
			S\indices{^l_{f_2e_1e_2}}
		}
		\\[\medskipamount]&
		\multicolumn{1}{r}{
		\young(\btwo\bone\done\eone\etwo,\ctwo,\dtwo,\ftwo)
			x^{b_1}x^{b_2}x^{d_1}x^{e_1}x^{e_2}
			u^{c_2}v^{d_2}w^{f_2}
		}
		&=&0
		\\[\medskipamount]
		\bar g_{ij}\bar g_{kl}\bar g_{mn}&
		\multicolumn{1}{@{}l}{
			\bigl(
				S\indices{^{ik}_{b_1b_2}}S\indices{^j_{c_2d_1d_2}}+
				S\indices{^i_{c_2b_1b_2}}S\indices{^j_{d_1}^k_{d_2}}
			\bigr)
			S\indices{^m_{f_2e_1e_2}}
			S\indices{^{nl}_{g_1g_2}}
			\qquad
		}
		\\[\medskipamount]&
		\multicolumn{1}{r}{
			\young(\btwo\bone\done\eone\etwo\gone\gtwo,\ctwo,\dtwo,\ftwo)
			x^{b_1}x^{b_2}x^{d_1}x^{e_1}x^{e_2}x^{g_1}x^{g_2}
			u^{c_2}v^{d_2}w^{f_2}
		}
		&=&0
		\\&
		\multicolumn{1}{r}{
			\forall x\in M
			\quad
			\forall u,v,w\in T_xM
			\quad
		}
	\end{array}
\]
respectively.  As before, the restrictions $\forall u,v,w\in T_xM$ and $\forall x\in M$ can be dropped and this allows us to
write both conditions independently of the vectors $x,u,v,w$ as
\begin{alignat}{3}
	\label{eq:2:alg}
	{\young(\btwo\bone\done\eone\etwo,\ctwo,\dtwo,\ftwo)}^\adjoint
	\bar g_{ij}
	\bar g_{kl}&
	\bigl(
		S\indices{^{ik}_{b_1b_2}}S\indices{^j_{c_2d_1d_2}}+&
		S\indices{^i_{c_2b_1b_2}}S\indices{^j_{d_1}^k_{d_2}}
	\bigr)&
	S\indices{^l_{f_2e_1e_2}}
	=0\\
	\label{eq:3:alg}
	{\young(\btwo\bone\done\eone\etwo\gone\gtwo,\ctwo,\dtwo,\ftwo)}^\adjoint
	\bar g_{ij}
	\bar g_{kl}
	\bar g_{mn}&
	\bigl(
		S\indices{^{ik}_{b_1b_2}}S\indices{^j_{c_2d_1d_2}}+&
		S\indices{^i_{c_2b_1b_2}}S\indices{^j_{d_1}^k_{d_2}}
	\bigr)&
	S\indices{^m_{f_2e_1e_2}}
	S\indices{^{nl}_{g_1g_2}}
	=0
	\fullstop
\end{alignat}
In order to simplify these conditions we need the following two lemmas.
\begin{lemma}
	\label{lem:equivalent}
	The first integrability condition is equivalent to
	\begin{equation}
		\label{eq:equivalent}
		\young(\btwo,\ctwo,\dtwo)
		\young(\bone\done)
		\bar g_{ij}
		\bigl(
			  S\indices{^i_{           b_1}^k_{\underline b_2}}
			+2S\indices{^i_{\underline b_2}^k_{           b_1}}
		\bigr)
		S\indices{^j_{\underline c_2d_1\underline d_2}}
		=0
		\fullstop
	\end{equation}
\end{lemma}
\begin{proof}
	Take the first integrability condition in the form \eqref{eq:KS:c} and reduce the antisymmetriser by the label $a_2$:
	\begin{align*}
		\smash[b]{\young(\btwo,\ctwo,\dtwo)}
		\bar g_{ij}
		\bigl(
			\phantom
			+S\indices{^i_{           a_2b_1\underline b_2}}S\indices{^j_{\underline c_2d_1\underline d_2}}&
			-S\indices{^i_{\underline b_2b_1\underline c_2}}S\indices{^j_{\underline d_2d_1           a_2}}\\
			+S\indices{^i_{\underline c_2b_1\underline d_2}}S\indices{^j_{           a_2d_1\underline b_2}}&
			-S\indices{^i_{\underline d_2b_1           a_2}}S\indices{^j_{\underline b_2d_1\underline c_2}}
		\bigr)
		=0
		\fullstop
	\end{align*}
	If we symmetrise this expression in $b_1,d_1$, the first and third as well as the second and fourth term become equal.
	Permuting indices, we get
	\begin{equation}
		\label{eq:intermediate}
		\young(\btwo,\ctwo,\dtwo)
		\young(\bone\done)
		\bar g_{ij}
		\bigl(
			S\indices{^i_{a_2b_1\underline b_2}}
			-S\indices{^i_{\underline b_2a_2b_1}}
		\bigr)
		S\indices{^j_{\underline c_2d_1\underline d_2}}
		=0
		\fullstop
	\end{equation}
	If we now symmetrise in $a_2,b_1,d_1$ and apply the symmetrised \name{Bianchi} identity to $S\indices{^i_{a_2b_1\underline
	b_2}}$, we get back the first integrability condition in the form \eqref{eq:KS:b}.  This proves its equivalence to
	\eqref{eq:intermediate}.  Applying now the \name{Bianchi} identity to the first term in \eqref{eq:intermediate} yields
	\eqref{eq:equivalent} with the index $k$ lowered and renamed as $a_2$.  
\end{proof}
\begin{lemma}
	\label{lem:consequence}
	The following identity is a consequence of the first integrability condition:
	\begin{equation}
		\label{eq:consequence}
		\young(\ctwo,\dtwo)
		\young(\btwo\bone\done)
		\bar g_{ij}
		\bigl(
					 S\indices{^i_{b_1}^k_{           b_2}}S\indices{^j_{\underline c_2d_1\underline d_2}}
			-\tfrac12S\indices{^i_{d_1}^k_{\underline d_2}}S\indices{^j_{\underline c_2b_1           b_2}}
			-        S\indices{^i_{\underline d_2}^k_{d_1}}S\indices{^j_{\underline c_2b_1           b_2}}
		\bigr)
		=0
		\fullstop
	\end{equation}
\end{lemma}
\begin{proof}
	Reduce the antisymmetriser in \eqref{eq:equivalent} by the index $b_2$,
	\begin{alignat*}{4}
		\smash[b]{\young(\ctwo,\dtwo)}
		\young(\bone\done)
		\bar g_{ij}&
		\bigl(
			& &&S\indices{^i_{b_1}^k_{           b_2}}S\indices{^j_{\underline c_2d_1\underline d_2}}&+&2&S\indices{^i_{           b_2}^k_{b_1}}S\indices{^j_{\underline c_2d_1\underline d_2}}\\&
			&+&&S\indices{^i_{b_1}^k_{\underline c_2}}S\indices{^j_{\underline d_2d_1           b_2}}&+&2&S\indices{^i_{\underline c_2}^k_{b_1}}S\indices{^j_{\underline d_2d_1           b_2}}\\&
			&+&&S\indices{^i_{b_1}^k_{\underline d_2}}S\indices{^j_{           b_2d_1\underline c_2}}&+&2&S\indices{^i_{\underline d_2}^k_{b_1}}S\indices{^j_{           b_2d_1\underline c_2}}
		\bigr)
		=0
		\comma
	\end{alignat*}
	and then symmetrise in $b_2,b_1,d_1$.  In the last line we can then apply the symmetrised \name{Bianchi} identity in order
	to move the antisymmetrised index $c_2$ from the fourth to the second position:
	\begin{alignat*}{4}
		\smash[b]{\young(\ctwo,\dtwo)}
		\young(\btwo\bone\done)
		\bar g_{ij}&
		\bigl(
			& &        &S\indices{^i_{b_1}^k_{           b_2}}S\indices{^j_{\underline c_2d_1\underline d_2}}&+&2&S\indices{^i_{           b_2}^k_{b_1}}S\indices{^j_{\underline c_2d_1\underline d_2}}\\&
			&-&        &S\indices{^i_{b_1}^k_{\underline d_2}}S\indices{^j_{\underline c_2d_1           b_2}}&-&2&S\indices{^i_{\underline d_2}^k_{b_1}}S\indices{^j_{\underline c_2d_1           b_2}}\\&
			&-&\tfrac12&S\indices{^i_{b_1}^k_{\underline d_2}}S\indices{^j_{\underline c_2d_1           b_2}}&-& &S\indices{^i_{\underline d_2}^k_{b_1}}S\indices{^j_{\underline c_2d_1           b_2}}
		\bigr)
		=0
		\fullstop
	\end{alignat*}
	After permuting indices under symmetrisation appropriately, we get the desired result.
\end{proof}
\begin{proposition}[Second integrability condition]
	\label{prop:KS2:4-4+}
	Suppose a \name{Killing} tensor on a constant sectional curvature manifold satisfies the first integrability condition
	\eqref{eq:TNS:1}.  Then the following conditions are equivalent to the second integrability condition \eqref{eq:TNS:2}:
	\begin{enumerate}
		\item
			\label{it:KS2:hook}
			The corresponding symmetrised algebraic curvature tensor $S$ satisfies one of the following two equivalent
			conditions:
			\begin{subequations}
				\label{eq:KS2:hook}
				\begin{alignat}{3}
					\label{eq:KS2:hook:yin}
					{\young(\btwo\bone\done\eone\etwo,\ctwo,\dtwo,\ftwo)}^\adjoint
					&\bar g_{ij}
					 \bar g_{kl}
					&S\indices{^i_{c_2d_1d_2}}
					&S\indices{^j_{b_1}^k_{b_2}}
					&S\indices{^l_{f_2e_1e_2}}
					&=0\\
					\label{eq:KS2:hook:yang}
					{\young(\btwo\bone\done\eone\etwo,\ctwo,\dtwo,\ftwo)}^\adjoint
					&\bar g_{ij}
					 \bar g_{kl}
					&S\indices{^i_{c_2b_1b_2}}
					&S\indices{^j_{d_1}^k_{d_2}}
					&S\indices{^l_{f_2e_1e_2}}
					&=0
					\fullstop
				\end{alignat}
			\end{subequations}
		\item
			\label{it:KS2:3-5+}
			The corresponding symmetrised algebraic curvature tensor $S$ satisfies one of the following two equivalent conditions:
			\begin{subequations}
				\label{eq:KS2:3-5+}
				\begin{align}
					\label{eq:KS2:3-5+:yin}
					\young(\ctwo,\dtwo,\ftwo)
					\young(\btwo\bone\done\eone\etwo)
					\bar g_{ij}
					\bar g_{kl}
					S\indices{^i_{\underline c_2d_1\underline d_2}}
					S\indices{^j_{b_1}^k_{b_2}}
					S\indices{^l_{\underline f_2e_1 e_2}}
					&=0\\
					\label{eq:KS2:3-5+:yang}
					\young(\ctwo,\dtwo,\ftwo)
					\young(\btwo\bone\done\eone\etwo)
					\bar g_{ij}
					\bar g_{kl}
					S\indices{^i_{\underline c_2b_1b_2}}
					S\indices{^j_{d_1}^k_{\underline d_2}}
					S\indices{^l_{\underline f_2e_1 e_2}}
					&=0
					\fullstop
				\end{align}
			\end{subequations}
		\item
			\label{it:KS2:4-4+}
			The corresponding symmetrised algebraic curvature tensor $S$ satisfies one of the following three equivalent
			conditions:
			\begin{subequations}
				\label{eq:KS2:4-4+}
				\begin{align}
					\label{eq:KS2:4-4+:left} &\bar g_{ij}\bar g_{kl}S\indices{^i_{\underline c_2d_1\underline d_2}}S\indices{^j_{           b_1}^k_{\underline b_2}}S\indices{^l_{\underline f_2e_1           e_2}}=0\\
					\label{eq:KS2:4-4+:right}\smash{\young(\btwo,\ctwo,\dtwo,\ftwo)\young(\bone\done\eone\etwo)\left\{\rule{0pt}{32pt}\right.}
					                         &\bar g_{ij}\bar g_{kl}S\indices{^i_{\underline c_2d_1\underline d_2}}S\indices{^j_{\underline b_2}^k_{           b_1}}S\indices{^l_{\underline f_2e_1           e_2}}=0\\
					\label{eq:KS2:4-4+:both} &\bar g_{ij}\bar g_{kl}S\indices{^i_{\underline c_2d_1\underline d_2}}S\indices{^j_{           e_1}^k_{           e_2}}S\indices{^l_{\underline f_2b_1\underline b_2}}=0
					\fullstop
				\end{align}
			\end{subequations}
		\item
			\label{it:KR:2}
			The corresponding algebraic curvature tensor $R$ satisfies
			\begin{equation}
				\label{eq:KR:2}
				\young(\btwo,\ctwo,\dtwo,\ftwo)
				\young(\bone\done\eone\etwo)
				\bar g_{ij}\bar g_{kl}
				R\indices{^i_{d_1\underline c_2\underline d_2}}
				R\indices{^j_{e_1}^k_{e_2}}
				R\indices{^l_{b_1\underline f_2\underline b_2}}
				=0
				\fullstop
			\end{equation}
	\end{enumerate}
\end{proposition}
\begin{remark}
	To facilitate the reading of this and subsequent proofs, note that the names of symmetrised indices are completely
	irrelevant.
\end{remark}
\begin{proof}
	\ref{it:KS2:hook}
	Contract \eqref{eq:consequence} with $\bar g_{kl}S\indices{^l_{f_2e_1e_2}}$, antisymmetrise in $c_2,d_2,f_2$ and symmetrise
	in $b_2,b_1,d_1,e_1,e_2$.  This yields
	\begin{align*}
		\smash[b]{\young(\ctwo,\dtwo,\ftwo)}&
		\young(\btwo\bone\done\eone\etwo)
		\bar g_{ij}\bar g_{kl}
		\\
		&\qquad
		\bigl(
			         S\indices{^i_{b_1}^k_{           b_2}}S\indices{^j_{\underline c_2d_1\underline d_2}}
			-\tfrac12S\indices{^i_{d_1}^k_{\underline d_2}}S\indices{^j_{\underline c_2b_1           b_2}}
			-        S\indices{^i_{\underline d_2}^k_{d_1}}S\indices{^j_{\underline c_2b_1           b_2}}
		\bigr)
		S\indices{^l_{\underline f_2e_1e_2}}
		=0
		\fullstop
	\end{align*}
	Reordering indices shows that the third term differs from the second by a factor of minus two.  Indeed,	exchanging $d_1$ and
	$d_2$ in the third term is tantamount to exchanging the upper indices $i$ and $k$, due to the pair symmetry of
	$S\indices{^i_{d_2}^k_{d_1}}$.  But under contraction with $\bar g_{ij}\bar g_{kl}$ this is tantamount to exchanging the
	upper indices $j$ and $l$.  This in turn is tantamount to exchanging $c_2,b_1,b_2$ with $f_2,e_1,e_2$ which, under
	symmetrisation and antisymmetrisation, is tantamount to a sign change.  Therefore
	\begin{equation}
		\label{eq:3-5+:intermediate}
		\young(\ctwo,\dtwo,\ftwo)
		\young(\btwo\bone\done\eone\etwo)
		\bar g_{ij}\bar g_{kl}
		\bigl(
			         S\indices{^i_{b_1}^k_{           b_2}}S\indices{^j_{\underline c_2d_1\underline d_2}}
			+\tfrac12S\indices{^i_{d_1}^k_{\underline d_2}}S\indices{^j_{\underline c_2b_1           b_2}}
		\bigr)
		S\indices{^l_{\underline f_2e_1e_2}}
		=0
		\fullstop
	\end{equation}
	Applying the symmetrised \name{Bianchi} identity to $S\indices{^i_{b_1}^k_{b_2}}$ and antisymmetrising in $b_2,c_2,d_2,f_2$
	yields
	\[
		{\young(\btwo\bone\done\eone\etwo,\ctwo,\dtwo,\ftwo)}^\adjoint
		\bar g_{ij}\bar g_{kl}
		\bigl(
			S\indices{^{ik}_{b_1b_2}}S\indices{^j_{c_2d_1d_2}}-
			S\indices{^i_{d_1}^k_{d_2}}S\indices{^j_{c_2b_1b_2}}
		\bigr)
		S\indices{^l_{f_2e_1e_2}}
		=0
		\fullstop
	\]
	We have derived this identity from the first integrability condition via \linkref{lemma}{lem:consequence}.  Comparing it
	with condition \eqref{eq:2:alg} shows that \eqref{eq:2:alg} is equivalent to \eqref{eq:KS2:hook:yang} and, after using once
	again the symmetrised \name{Bianchi} identity, also to \eqref{eq:KS2:hook:yin}.  This proves \ref{it:KS2:hook}, since we
	have already shown that the second integrability condition is equivalent to \eqref{eq:2:alg}.  

	\ref{it:KS2:3-5+}
	Condition \eqref{eq:KS2:3-5+:yin} is equivalent to \eqref{eq:KS2:hook:yin}.  This results from \eqref{eq:hooks} when taking
	\eqref{eq:hook=0:2:yin} and \eqref{eq:kernels} into account.  In the same way \eqref{eq:KS2:3-5+:yang} is equivalent to
	\eqref{eq:KS2:hook:yang}, using \eqref{eq:hook=0:2:yang}.

	\ref{it:KS2:4-4+}
	We will prove the equivalence of \eqref{eq:KS2:hook:yin} to each of the equations \eqref{eq:KS2:4-4+}.  To this aim we
	establish three linearly independent homogeneous equations for the three tensors on the left hand side of
	\eqref{eq:KS2:4-4+}.
	\begin{subequations}
		\label{eq:system:2}
		For the first equation we split the hook symmetriser in \eqref{eq:KS2:hook:yin} at the label $b_2$ and get
		\begin{alignat*}{4}
			\smash[b]{\young(\btwo,\ctwo,\dtwo,\ftwo)}
			\young(\bone\done\eone\etwo)
			&\bar g_{ij}\bar g_{kl}
			&\bigl(
				& &S\indices{^i_{\underline c_2           d_1\underline d_2}}S\indices{^j_{           b_1}^k_{\underline b_2}}S\indices{^l_{\underline f_2           e_1           e_2}}&\\&&
				&+&S\indices{^i_{\underline c_2           e_1\underline d_2}}S\indices{^j_{           d_1}^k_{           b_1}}S\indices{^l_{\underline f_2           e_2\underline b_2}}&\\&&
				&+&S\indices{^i_{\underline c_2           e_2\underline d_2}}S\indices{^j_{           e_1}^k_{           d_1}}S\indices{^l_{\underline f_2\underline b_2           b_1}}&\\&&
				&+&S\indices{^i_{\underline c_2\underline b_2\underline d_2}}S\indices{^j_{           e_2}^k_{           e_1}}S\indices{^l_{\underline f_2           b_1           d_1}}&\\&&
				&+&S\indices{^i_{\underline c_2           b_1\underline d_2}}S\indices{^j_{\underline b_2}^k_{           e_2}}S\indices{^l_{\underline f_2           d_1           e_1}}&
			\bigr)
			=0
			\fullstop
		\end{alignat*}
		The fourth term vanishes by the \name{Bianchi} identity and the second term is equal to the third.  Therefore
		\eqref{eq:KS2:hook:yin} is equivalent to
		\begin{equation}
			\label{eq:system:2:two}
			\begin{split}
				\smash[b]{\young(\btwo,\ctwo,\dtwo,\ftwo)}
				&
				\young(\bone\done\eone\etwo)
				\bar g_{ij}\bar g_{kl}
				S\indices{^i_{\underline c_2d_1\underline d_2}}\\
				&\qquad
				\bigl(
					  S\indices{^j_{           b_1}^k_{\underline b_2}}S\indices{^l_{\underline f_2e_1           e_2}}
					+2S\indices{^j_{           e_1}^k_{           e_2}}S\indices{^l_{\underline f_2b_1\underline b_2}}
					+ S\indices{^j_{\underline b_2}^k_{           b_1}}S\indices{^l_{\underline f_2e_1           e_2}}
				\bigr)
				=0
				\fullstop
			\end{split}
		\end{equation}
		This is our first equation.  The other two equations follow from the first integrability condition as follows.  The
		second equation is obtained from \eqref{eq:equivalent} by contracting with $\bar g_{kl}S\indices{^l_{f_2e_1e_2}}$,
		antisymmetrising in $b_2,c_2,d_2,f_2$ and symmetrising in $b_1,d_1,e_1,e_2$:
		\[
			\young(\btwo,\ctwo,\dtwo,\ftwo)
			\young(\bone\done\eone\etwo)
			\bar g_{ij}\bar g_{kl}
			\bigl(
				  S\indices{^i_{           b_1}^k_{\underline b_2}}
				+2S\indices{^i_{\underline b_2}^k_{           b_1}}
			\bigr)
			S\indices{^j_{\underline c_2d_1\underline d_2}}
			S\indices{^l_{\underline f_2e_1           e_2}}
			=0
			\fullstop
		\]
		This can be rewritten as
		\begin{equation}
			\label{eq:system:2:equivalent}
			\young(\btwo,\ctwo,\dtwo,\ftwo)
			\young(\bone\done\eone\etwo)
			\bar g_{ij}\bar g_{kl}
			S\indices{^i_{\underline c_2d_1\underline d_2}}
			\bigl(
				  S\indices{^j_{           b_1}^k_{\underline b_2}}S\indices{^l_{\underline f_2e_1e_2}}
				+2S\indices{^j_{\underline b_2}^k_{           b_1}}S\indices{^l_{\underline f_2e_1e_2}}
			\bigr)
			=0
		\end{equation}
		and is our second equation.  For the third equation, we rename $b_1,b_2$ in \eqref{eq:consequence} as $e_1,e_2$,
		contract with $\bar g_{kl}S\indices{^l_{f_2b_1b_2}}$, antisymmetrise in $b_2,c_2,d_2,f_2$ and symmetrise in
		$b_1,d_1,e_1,e_2$:
		\begin{align*}
			\smash[b]{\young(\btwo,\ctwo,\dtwo,\ftwo)}
			&
			\young(\bone\done\eone\etwo)
			\bar g_{ij}
			\bar g_{kl}
			S\indices{^l_{\underline f_2b_1\underline b_2}}\\
			&\qquad
			\bigl(
						 S\indices{^i_{e_1}^k_{           e_2}}S\indices{^j_{\underline c_2d_1\underline d_2}}
				-\tfrac12S\indices{^i_{d_1}^k_{\underline d_2}}S\indices{^j_{\underline c_2e_1           e_2}}
				-        S\indices{^i_{\underline d_2}^k_{d_1}}S\indices{^j_{\underline c_2e_1           e_2}}
			\bigr)
			=0
			\fullstop
		\end{align*}
		This can be rewritten as
		\begin{equation}
			\label{eq:system:2:consequence}
			\begin{split}
				\smash[b]{\young(\btwo,\ctwo,\dtwo,\ftwo)}
				&
				\young(\bone\done\eone\etwo)
				\bar g_{ij}
				\bar g_{kl}
				S\indices{^i_{\underline c_2d_1\underline d_2}}\\
				&\qquad
				\bigl(
							 S\indices{^j_{           e_1}^k_{e_2}}S\indices{^l_{\underline f_2b_1\underline b_2}}
					-\tfrac12S\indices{^j_{\underline b_2}^k_{b_1}}S\indices{^l_{\underline f_2e_1           e_2}}
					-        S\indices{^j_{b_1}^k_{\underline b_2}}S\indices{^l_{\underline f_2e_1           e_2}}
				\bigr)
				=0
			\end{split}
		\end{equation}
		and is our last equation.
	\end{subequations}
	Clearly, the resulting homogeneous system \eqref{eq:system:2} implies \eqref{eq:KS2:4-4+}.  On the other hand, any of the
	equations \eqref{eq:KS2:4-4+} together with \eqref{eq:system:2:equivalent} and \eqref{eq:system:2:consequence} implies
	\eqref{eq:system:2:two} and therefore \eqref{eq:KS2:hook:yin}.

	\ref{it:KR:2}
	Condition \eqref{eq:KR:2} is equivalent to \eqref{eq:KS2:4-4+:both} via \eqref{eq:R<->S}.
\end{proof}

\subsection{Redundancy of the third integrability condition}

The aim of this section is to prove the following:
\begin{proposition}[Third integrability condition]
	For a \name{Killing} tensor on a constant sectional curvature manifold the third of the three integrability conditions
	\eqref{eq:TNS} is redundant.
\end{proposition}
We have already shown that the third integrability condition is equivalent to \eqref{eq:3:alg}.  As before we can infer from
\eqref{eq:hooks} together with \eqref{eq:hook=0:3} and \eqref{eq:kernels} that \eqref{eq:3:alg} is equivalent to
\begin{equation}
	\label{eq:3-7+}
	\begin{split}
		&
		\smash[b]{\young(\ctwo,\dtwo,\ftwo)}
		\young(\btwo\bone\done\eone\etwo\gone\gtwo)
		\bar g_{ij}
		\bar g_{kl}
		\bar g_{mn}
		\\&\qquad
		\bigl(
			S\indices{^{ik}_{b_1b_2}}S\indices{^j_{\underline c_2d_1\underline d_2}}+
			S\indices{^i_{\underline c_2b_1b_2}}S\indices{^j_{d_1}^k_{\underline d_2}}
		\bigr)
		S\indices{^m_{\underline f_2e_1e_2}}
		S\indices{^{nl}_{g_1g_2}}
		=0
		\fullstop
	\end{split}
\end{equation}
The proceeding to prove this equation is similar to the proof of part \ref{it:KS2:4-4+} in \linkref{proposition}{prop:KS2:4-4+}.
From the first two integrability conditions we will deduce the following three equations
\begin{subequations}
	\label{eq:system:3}
	\begin{alignat}{5}
		\label{eq:system:3:one}     \eqref{eq:-++}&-&\tfrac12&\eqref{eq:+-+}&-&\eqref{eq:++-}&=&0\\
		\label{eq:system:3:onexone}               & &      2 &\eqref{eq:+-+}&+&\eqref{eq:++-}&=&0\\
		\label{eq:system:3:two}                   & &        &\eqref{eq:+-+}&-&\eqref{eq:++-}&=&0
	\end{alignat}
\end{subequations}
for the tensors
\begin{subequations}
	\label{eq:+-}
	\begin{align}
		\label{eq:-++}&\bar g_{ij}\bar g_{kl}\bar g_{mn}S\indices{^i_{\underline c_2d_1\underline d_2}}S\indices{^j_{b_1}^k_{b_2}}S\indices{^l_{g_1}^m_{g_2}}S\indices{^n_{\underline f_2e_1e_2}}\\
		\label{eq:+-+}\smash{\young(\ctwo,\dtwo,\ftwo)\young(\btwo\bone\done\eone\etwo\gone\gtwo)\left\{\rule{0pt}{32pt}\right.}
		              &\bar g_{ij}\bar g_{kl}\bar g_{mn}S\indices{^i_{\underline c_2b_1b_2}}S\indices{^j_{d_1}^k_{\underline d_2}}S\indices{^l_{g_1}^m_{g_2}}S\indices{^n_{\underline f_2e_1e_2}}\\
		\label{eq:++-}&\bar g_{ij}\bar g_{kl}\bar g_{mn}S\indices{^i_{\underline c_2b_1b_2}}S\indices{^j_{\underline d_2}^k_{d_1}}S\indices{^l_{g_1}^m_{g_2}}S\indices{^n_{\underline f_2e_1e_2}}
		\fullstop
	\end{align}
\end{subequations}
The system \eqref{eq:system:3} shows that each of the tensors \eqref{eq:+-} is zero.  In particular this proves our claim, since
\eqref{eq:3-7+} can be written as a linear combination of these tensors.

\subsubsection{First equation}
Contract \eqref{eq:consequence} with $\bar g_{kl}\bar g_{mn}S\indices{^l_{g_1}^m_{g_2}}S\indices{^n_{f_2e_1e_2}}$,
antisymmetrise in $c_2,d_2,f_2$ and symmetrise in the remaining seven indices:
\begin{align*}
	\smash[b]{\young(\ctwo,\dtwo,\ftwo)}&
	\young(\btwo\bone\done\eone\etwo\gone\gtwo)
	\bar g_{ij}
	\bar g_{kl}
	\bar g_{mn}
	\\
	&\qquad
	\bigl(
				 S\indices{^i_{b_1}^k_{           b_2}}S\indices{^j_{\underline c_2d_1\underline d_2}}
		-\tfrac12S\indices{^i_{d_1}^k_{\underline d_2}}S\indices{^j_{\underline c_2b_1           b_2}}
		-        S\indices{^i_{\underline d_2}^k_{d_1}}S\indices{^j_{\underline c_2b_1           b_2}}
	\bigr)
	S\indices{^l_{g_1}^m_{g_2}}
	S\indices{^n_{\underline f_2e_1e_2}}
	=0
	\fullstop
\end{align*}
Renaming $i,j$ as $j,i$, this can be written as
\begin{align*}
	\smash[b]{\young(\ctwo,\dtwo,\ftwo)}&
	\young(\btwo\bone\done\eone\etwo\gone\gtwo)
	\bar g_{ij}
	\bar g_{kl}
	\bar g_{mn}
	\\
	&\qquad
	\bigl(
				 S\indices{^i_{\underline c_2d_1\underline d_2}}S\indices{^j_{b_1}^k_{           b_2}}
		-\tfrac12S\indices{^i_{\underline c_2b_1           b_2}}S\indices{^j_{d_1}^k_{\underline d_2}}
		-        S\indices{^i_{\underline c_2b_1           b_2}}S\indices{^j_{\underline d_2}^k_{d_1}}
	\bigr)
	S\indices{^l_{g_1}^m_{g_2}}
	S\indices{^n_{\underline f_2e_1e_2}}
	=0
	\fullstop
\end{align*}
This is our first equation \eqref{eq:system:3:one}.

\subsubsection{Second equation}
Reduce the antisymmetriser in \eqref{eq:consequence} completely,
\begin{align*}
	\young(\btwo\bone\done)
	\bar g_{ij}
	\bigl(
		 &        S\indices{^i_{b_1}^k_{b_2}}S\indices{^j_{c_2d_1d_2}}
		- \tfrac12S\indices{^i_{d_1}^k_{d_2}}S\indices{^j_{c_2b_1b_2}}
		-         S\indices{^i_{d_2}^k_{d_1}}S\indices{^j_{c_2b_1b_2}}\\
		-&        S\indices{^i_{b_1}^k_{b_2}}S\indices{^j_{d_2d_1c_2}}
		+ \tfrac12S\indices{^i_{d_1}^k_{c_2}}S\indices{^j_{d_2b_1b_2}}
		+         S\indices{^i_{c_2}^k_{d_1}}S\indices{^j_{d_2b_1b_2}}
	\bigr)
	=0
	\comma
\end{align*}
raise the index $c_2$, rename it as $m$ and bring the last two terms to the right hand side,
\begin{align*}
	&\young(\btwo\bone\done)
	\bar g_{ij}
	\bigl(
		         S\indices{^i_{b_1}^k_{b_2}}S\indices{^{jm}_{d_1d_2}}
		-\tfrac12S\indices{^i_{d_1}^k_{d_2}}S\indices{^{jm}_{b_1b_2}}
		-        S\indices{^i_{d_2}^k_{d_1}}S\indices{^{jm}_{b_1b_2}}
		-        S\indices{^i_{b_1}^k_{b_2}}S\indices{^j_{d_2d_1}^m}
	\bigr)                                
	\\&\qquad=                            
	-\young(\btwo\bone\done)
	\bar g_{ij}
	\bigl(
		\tfrac12S\indices{^i_{d_1}^{km}}S\indices{^j_{d_2b_1b_2}}
		+       S\indices{^{imk}_{d_1}} S\indices{^j_{d_2b_1b_2}}
	\bigr)
	\comma
\end{align*}
contract with $\bar g_{kl}\bar g_{mn}S\indices{^l_{f_2e_1e_2}}S\indices{^n_{h_2g_1g_2}}$, antisymmetrise in $d_2,f_2,h_2$ and
symmetrise in the remaining seven indices:
\[
	\begin{split}
		&
		\begin{split}
			\smash[b]{\young(\dtwo,\ftwo,\htwo)}&
			\young(\btwo\bone\done\eone\etwo\gone\gtwo)
			\bar g_{ij}
			\bar g_{kl}
			\bar g_{mn}
			S\indices{^l_{\underline f_2e_1e_2}}
			S\indices{^n_{\underline h_2g_1g_2}}
			\\&\qquad
			\bigl(
						 S\indices{^i_{           b_1}^k_{           b_2}}S\indices{^{jm}_{d_1\underline d_2}}
				-\tfrac12S\indices{^i_{           d_1}^k_{\underline d_2}}S\indices{^{jm}_{b_1b_2}}
				-        S\indices{^i_{\underline d_2}^k_{           d_1}}S\indices{^{jm}_{b_1b_2}}
				-        S\indices{^i_{           b_1}^k_{           b_2}}S\indices{^j_{\underline d_2d_1}^m}
			\bigr)
		\end{split}
		\\&\qquad
		\begin{split}
			=-\smash[b]{\young(\dtwo,\ftwo,\htwo)}&
			\young(\btwo\bone\done\eone\etwo\gone\gtwo)
			\bar g_{ij}
			\bar g_{kl}
			\bar g_{mn}
			\\&\qquad
			\bigl(
				\tfrac12S\indices{^i_{d_1}^{km}}S\indices{^j_{\underline d_2b_1b_2}}
				+       S\indices{^{imk}_{d_1}} S\indices{^j_{\underline d_2b_1b_2}}
			\bigr)
			S\indices{^l_{\underline f_2e_1e_2}}
			S\indices{^n_{\underline h_2g_1g_2}}
			\fullstop
		\end{split}
	\end{split}
\]
On the right hand side the upper indices $j,n,l$ are implicitely antisymmetrised by the symmetrisation in
$b_1,b_2,e_1,e_2,g_1,g_2$ and the antisymmetrisation in $d_2,f_2,h_2$.  Due to the term $\bar g_{ij}\bar g_{kl}\bar g_{mn}$ the
same holds for the upper indices $i,m,k$.  The \name{Bianchi} identity therefore implies that the right hand side is zero.  On
the left hand side, the \name{Bianchi} identity allows us to bring the index $m$ in each term to the third position:
\begin{alignat*}{3}
	\smash[b]{\young(\dtwo,\ftwo,\htwo)}
	\young(\btwo\bone\done\eone\etwo\gone\gtwo)
	&
	\bar g_{ij}
	\bar g_{kl}
	\bar g_{mn}
	\\
	S\indices{^l_{\underline f_2e_1e_2}}
	\bigl(
		-  &S\indices{^i_{b_1}^k_{           b_2}}S\indices{^j_{d_1}^m_{\underline d_2}}
		+& &S\indices{^i_{d_1}^k_{\underline d_2}}S\indices{^j_{b_1}^m_{           b_2}}
	\\
		- 2&S\indices{^i_{           b_1}^k_{b_2}}S\indices{^j_{\underline d_2}^m_{d_1}}
		+&2&S\indices{^i_{\underline d_2}^k_{d_1}}S\indices{^j_{           b_1}^m_{b_2}}
	\bigr)                                
	S\indices{^n_{\underline h_2g_1g_2}}
	=0
	\fullstop
\end{alignat*}
Using pair symmetry and renaming the indices $i,j,k,l$ as $k,l,j,i$, this can be written as
\begin{alignat*}{3}
	\smash[b]{\young(\dtwo,\ftwo,\htwo)}
	\young(\btwo\bone\done\eone\etwo\gone\gtwo)
	&
	\bar g_{ij}
	\bar g_{kl}
	\bar g_{mn}
	\\
	S\indices{^i_{\underline f_2e_1e_2}}
	\bigl(
		-  &S\indices{^j_{           b_2}^k_{b_1}}S\indices{^l_{d_1}^m_{\underline d_2}}
		+& &S\indices{^j_{\underline d_2}^k_{d_1}}S\indices{^l_{b_1}^m_{           b_2}}
	\\
		- 2&S\indices{^j_{b_2}^k_{           b_1}}S\indices{^l_{\underline d_2}^m_{d_1}}
		+&2&S\indices{^j_{d_1}^k_{\underline d_2}}S\indices{^l_{           b_1}^m_{b_2}}
	\bigr)                                
	S\indices{^n_{\underline h_2g_1g_2}}
	=0
\end{alignat*}
Renaming $i,j,k,l,m,n$ in reverse order, the first term can be seen to differ from the second by a sign.  The same is true for
the third and fourth term, resulting in
\begin{align*}
	\smash[b]{\young(\dtwo,\ftwo,\htwo)}&
	\young(\btwo\bone\done\eone\etwo\gone\gtwo)
	\bar g_{ij}
	\bar g_{kl}
	\bar g_{mn}
	\\
	&\qquad
	S\indices{^i_{\underline f_2e_1e_2}}
	\bigl(
		  S\indices{^j_{\underline d_2}^k_{d_1}}S\indices{^l_{b_1}^m_{           b_2}}
		+2S\indices{^j_{d_1}^k_{\underline d_2}}S\indices{^l_{           b_1}^m_{b_2}}
	\bigr)                                
	S\indices{^n_{\underline h_2g_1g_2}}
	=0
	\fullstop
\end{align*}
This is our second equation \eqref{eq:system:3:onexone}.

\subsubsection{Third equation}
Take the second integrability condition in the form \eqref{eq:KS2:3-5+:yang}, reduce the antisymmetriser by the label $f_2$,
\begin{align*}
	\young(\ctwo,\dtwo)\young(\btwo\bone\done\eone\etwo)
	\bar g_{ij}\bar g_{kl}
	\bigl(
		 &S\indices{^i_{\underline c_2b_1b_2}}S\indices{^j_{d_1}^k_{\underline d_2}}S\indices{^l_{           f_2e_1e_2}}\\
		+&S\indices{^i_{\underline d_2b_1b_2}}S\indices{^j_{d_1}^k_{           f_2}}S\indices{^l_{\underline c_2e_1e_2}}\\
		+&S\indices{^i_{           f_2b_1b_2}}S\indices{^j_{d_1}^k_{\underline c_2}}S\indices{^l_{\underline d_2e_1e_2}}
	\bigr)
	=0
	\comma
\end{align*}
raise the index $f_2$, rename it as $m$ and bring the second term to the right hand side,
\begin{align*}
	\young(\ctwo,\dtwo)&
	\young(\btwo\bone\done\eone\etwo)
	\bar g_{ij}\bar g_{kl}
	\bigl(
		S\indices{^i_{\underline c_2b_1b_2}}
		S\indices{^j_{d_1}^k_{\underline d_2}}
		S\indices{^{lm}_{e_1e_2}}
		+
		S\indices{^{im}_{b_1b_2}}
		S\indices{^j_{d_1}^k_{\underline c_2}}
		S\indices{^l_{\underline d_2e_1e_2}}
	\bigr)
	\\&\qquad
	=-
	\young(\ctwo,\dtwo)
	\young(\btwo\bone\done\eone\etwo)
	\bar g_{ij}
	\bar g_{kl}
	S\indices{^i_{\underline d_2b_1b_2}}
	S\indices{^j_{d_1}^{km}}
	S\indices{^l_{\underline c_2e_1e_2}}
	\comma
\end{align*}
contract with $\bar g_{mn}S\indices{^n_{f_2g_1g_2}}$, antisymmetrise in $c_2,d_2,f_2$ and symmetrise in the remaining seven indices:
\begin{align*}
	\smash[b]{\young(\ctwo,\dtwo,\ftwo)}&
	\young(\btwo\bone\done\eone\etwo\gone\gtwo)
	\bar g_{ij}
	\bar g_{kl}
	\bar g_{mn}
	\\&\qquad
	\bigl(
		S\indices{^i_{\underline c_2b_1b_2}}
		S\indices{^j_{d_1}^k_{\underline d_2}}
		S\indices{^{lm}_{e_1e_2}}
		+
		S\indices{^{im}_{b_1b_2}}
		S\indices{^j_{d_1}^k_{\underline c_2}}
		S\indices{^l_{\underline d_2e_1e_2}}
	\bigr)
	S\indices{^n_{\underline f_2g_1g_2}}
	\\&
	=-
	\young(\ctwo,\dtwo,\ftwo)
	\young(\btwo\bone\done\eone\etwo\gone\gtwo)
	\bar g_{ij}
	\bar g_{kl}
	\bar g_{mn}
	S\indices{^i_{\underline d_2b_1b_2}}
	S\indices{^j_{d_1}^{km}}
	S\indices{^l_{\underline c_2e_1e_2}}
	S\indices{^n_{\underline f_2g_1g_2}}
	\fullstop
\end{align*}
As before, the right hand side is zero.  On the left hand side use the symmetrised \name{Bianchi} identity
to move the upper index $m$ to the third position:
\begin{align*}
	\smash[b]{\young(\ctwo,\dtwo,\ftwo)}&
	\young(\btwo\bone\done\eone\etwo\gone\gtwo)
	\bar g_{ij}
	\bar g_{kl}
	\bar g_{mn}
	\\&\qquad
	\bigl(
		S\indices{^i_{\underline c_2b_1b_2}}
		S\indices{^j_{d_1}^k_{\underline d_2}}
		S\indices{^l_{e_1}^m_{e_2}}
		+
		S\indices{^i_{b_1}^m_{b_2}}
		S\indices{^j_{d_1}^k_{\underline c_2}}
		S\indices{^l_{\underline d_2e_1e_2}}
	\bigr)
	S\indices{^n_{\underline f_2g_1g_2}}
	=0
	\fullstop
\end{align*}
Using pair symmetry and renaming the upper indices in the second term, this can be written as
\begin{align*}
	\smash[b]{\young(\ctwo,\dtwo,\ftwo)}&
	\young(\btwo\bone\done\eone\etwo\gone\gtwo)
	\bar g_{ij}
	\bar g_{kl}
	\bar g_{mn}
	\\&\qquad
	\bigl(
		S\indices{^i_{\underline c_2b_1b_2}}
		S\indices{^j_{d_1}^k_{\underline d_2}}
		S\indices{^l_{e_1}^m_{e_2}}
		+
		S\indices{^i_{\underline d_2e_1e_2}}
		S\indices{^j_{\underline c_2}^k_{d_1}}
		S\indices{^l_{b_1}^m_{b_2}}
	\bigr)
	S\indices{^n_{\underline f_2g_1g_2}}
	=0
	\fullstop
\end{align*}
This is our third and last equation \eqref{eq:system:3:two}.

\section{Application}
\label{sec:application}

Finally, we show that the family \eqref{eq:family} satisfies the algebraic integrability conditions \eqref{eq:main} and
therefore describes integrable \name{Killing} tensors.
\begin{proof}[Proof of the Main Corollary]
	We will write a dot in place of each index whose name is irrelevant for our considerations.  Written in components, the
	algebraic curvature tensor
	\[
		R=\lambda_2h\varowedge h+\lambda_1h\varowedge g+\lambda_0g\varowedge g
	\]
	is then a linear combination of tensors of the form $h_{\cdot\cdot}h_{\cdot\cdot}$, $h_{\cdot\cdot}g_{\cdot\cdot}$ and
	$g_{\cdot\cdot}g_{\cdot\cdot}$ or, written in another way, of the form $A_{\cdot\cdot}B_{\cdot\cdot}$ with $A,B\in\{g,h\}$.
	Then $R_{a_1b_1a_2b_2}R_{c_1d_1c_2d_2}$ is a linear combination of terms of the form
	$A_{\cdot\cdot}B_{\cdot\cdot}C_{\cdot\cdot}D_{\cdot\cdot}$ with $A,B,C,D\in\{g,h\}$ and thus
	$g_{ij}R\indices{^i_{b_1a_2b_2}}R\indices{^i_{d_1c_2d_2}}$ is a linear combination of terms of the form
	$A_{\cdot\cdot}B_{\cdot\cdot}C_{\cdot\cdot}$ with $A,B,C\in\{g,h,h^2\}$.  Here $g$, $h$ and $h^2$ are symmetric tensors,
	where $h^2$ denotes the tensor $g_{ij}h\indices{^i_a}h\indices{^j_b}$.  Therefore the antisymmetrisation of
	$A_{\cdot\cdot}B_{\cdot\cdot}C_{\cdot\cdot}$ in four of the six indices vanishes by \name{Dirichlet}'s drawer principle.
	This proves that the first integrability condition \eqref{eq:main:1} is satisfied.

	In the same way, the tensor $R_{\cdot\cdot\cdot\cdot}R_{\cdot\cdot\cdot\cdot}R_{\cdot\cdot\cdot\cdot}$ is a linear
	combination of terms of the form $A_{\cdot\cdot}B_{\cdot\cdot}C_{\cdot\cdot}D_{\cdot\cdot}E_{\cdot\cdot}F_{\cdot\cdot}$ with
	$A,\ldots,F\in\{g,h\}$.  Hence the tensor
	\[
		g_{ij}
		g_{kl}
		R\indices{^i_{b_1a_2b_2}}
		R\indices{^j_{a_1}^k_{c_1}}
		R\indices{^l_{d_1c_2d_2}}
	\]
	is a linear combination of terms of the form $A_{\cdot\cdot}B_{\cdot\cdot}C_{\cdot\cdot}D_{\cdot\cdot}$ with either
	$A,B,C,D\in\{g,h,h^2\}$ or $A,B,C,D\in\{g,h,h^3\}$, where $h^3$ denotes the symmetric tensor
	\[
		g_{ij}g_{kl}h\indices{^i_a}h^{jk}h\indices{^l_b}
		\,\fullstop
	\]
	Without loss of generality we may suppose $D=C$, owing to the drawer principle.  Consider therefore the tensor
	$A_{\cdot\cdot}B_{\cdot\cdot}C_{\cdot\cdot}C_{\cdot\cdot}$ under antisymmetrisation in four of its indices and
	symmetrisation in the remaing four.  The result vanishes trivially if the antisymmetrisation includes an index pair of one
	of the symmetric tensors $A,B,C$.  Otherwise it can be written as
	\[
		\young(\atwo,\btwo,\ctwo,\dtwo)
		\young(\aone \bone \cone \done)
		A_{a_1a_2}B_{b_1b_2}C_{c_1c_2}C_{d_1d_2}
	\]
	and vanishes too, which becomes evident when $c_1,c_2$ is renamed as $d_1,d_2$.  This demonstrates that the second algebraic
	integrability condition \eqref{eq:main:2} is also satisfied.
\end{proof}
\begin{remark}
	The family \eqref{eq:family} properly extends \name{Benenti} tensors, given by
	\begin{align*}
		\lambda_0=\lambda_1&=0&
		\lambda_2&=\tfrac12&
		h&=Ag&
		A&\in\GLG(V)
		\fullstop
	\end{align*}
	Indeed, a \name{Killing} tensor corresponding to $h\varowedge g$ with $\tr h<0$ for example is not a \name{Benenti} tensor,
	as can be seen by comparing the scalar curvatures
	\begin{align*}
		\operatorname{Scal}\bigl(\tfrac12(Ag)\varowedge(Ag)\bigr)
		&=\tr^2(Ag)-\tr(Ag)^2
		=\tr^2(A^T\!\!A)-\tr(A^T\!\!AA^T\!\!A)\\
		&=\lVert A\rVert^4-\lVert A^T\!\!A\rVert^2
		\ge0\\
	\intertext{and}
		\operatorname{Scal}(h\varowedge g)
		&=2(N-1)\tr h
		<0
		\fullstop
	\end{align*}
\end{remark}

\bibliographystyle{amsalpha}
\bibliography{\jobname}

\providecommand{\bysame}{\leavevmode\hbox to3em{\hrulefill}\thinspace}
\providecommand{\MR}{\relax\ifhmode\unskip\space\fi MR }
\providecommand{\MRhref}[2]{%
  \href{http://www.ams.org/mathscinet-getitem?mr=#1}{#2}
}
\providecommand{\href}[2]{#2}
\begin{thebibliography}{RWML99}

\bibitem[Ben92]{Benenti92}
Sergio Benenti, \emph{Inertia tensors and {S}täckel systems in the {E}uclidean
  spaces}, Rendiconti del Seminario Matematico Politecnico di Torino
  \textbf{50} (1992), no.~4, 315--341.

\bibitem[Ben93]{Benenti93}
\bysame, \emph{Orthogonal separable dynamical systems}, Differential geometry
  and its applications ({O}pava, 1992), Mathematical Publications, vol.~1,
  Silesian University Opava, Opava, 1993, pp.~163--184.

\bibitem[Ben05]{Benenti05}
\bysame, \emph{Special symmetric two-tensors, equivalent dynamical systems,
  cofactor and bi-cofactor systems}, Acta Applicandae Mathematicae \textbf{87}
  (2005), no.~1-3, 33--91.

\bibitem[BM03]{Bolsinov&Matveev}
Alexey~V. Bolsinov and Vladimir~S. Matveev, \emph{Geometrical interpretation of
  {B}enenti systems}, Journal of Geometry and Physics \textbf{44} (2003),
  no.~4, 489--506.

\bibitem[Cra03a]{Crampin03a}
M.~Crampin, \emph{Conformal {K}illing tensors with vanishing torsion and the
  separation of variables in the {H}amilton-{J}acobi equation}, Differential
  Geometry and its Applications \textbf{18} (2003), no.~1, 87--102.

\bibitem[Cra03b]{Crampin03b}
\bysame, \emph{Projectively equivalent {R}iemannian spaces as
  quasi-bi-{H}amiltonian systems}, Acta Applicandae Mathematicae. An
  International Survey Journal on Applying Mathematics and Mathematical
  Applications \textbf{77} (2003), no.~3, 237--248.

\bibitem[CS01]{Crampin&Sarlet01}
M.~Crampin and W.~Sarlet, \emph{A class of nonconservative {L}agrangian systems
  on {R}iemannian manifolds}, Journal of Mathematical Physics \textbf{42}
  (2001), no.~9, 4313--4326.

\bibitem[CS02]{Crampin&Sarlet02}
\bysame, \emph{Bi-quasi-{H}amiltonian systems}, Journal of Mathematical Physics
  \textbf{43} (2002), no.~5, 2505--2517.

\bibitem[CST00]{Crampin&Sarlet&Thompson}
M.~Crampin, W.~Sarlet, and G.~Thompson, \emph{Bi-differential calculi,
  bi-{H}amiltonian systems and conformal {K}illing tensors}, Journal of
  Physics. A. Mathematical and General \textbf{33} (2000), no.~48, 8755--8770.

\bibitem[Eis34]{Eisenhart}
Luther~P. Eisenhart, \emph{Separable systems of {S}täckel}, Annals of
  Mathematics. Second Series \textbf{35} (1934), no.~2, 284--305.

\bibitem[FKWC92]{Fulling&King&Wybourne&Cummins}
S.~A. Fulling, R.~C. King, B.~G. Wybourne, and C.~J. Cummins, \emph{Normal
  forms for tensor polynomials: {I}. {T}he {R}iemann tensor}, Classical Quantum
  Gravity \textbf{9} (1992), 1151--1197.

\bibitem[Ful97]{Fulton}
William Fulton, \emph{{Y}oung tableaux: With applications to representation
  theory and geometry}, London Mathematical Society Student Texts, no.~35,
  Cambridge University Press, 1997.

\bibitem[Har92]{Harris}
Joe Harris, \emph{Algebraic geometry}, Graduate Texts in Mathematics, no. 133,
  Springer-Verlag, 1992.

\bibitem[HMS05]{Horwood&McLenaghan&Smirnov}
Joshua~T. Horwood, Raymond~G. McLenaghan, and Roman~G. Smirnov, \emph{Invariant
  classification of orthogonally separable {H}amiltonian systems in {E}uclidean
  space}, Communications in Mathematical Physics \textbf{259} (2005), 670--709.

\bibitem[Hor07]{Horwood}
Joshua~T. Horwood, \emph{On the theory of algebraic invariants of vector spaces
  of {K}illing tensors}, Journal of Geometry and Physics \textbf{58} (2007),
  487--501.

\bibitem[IMM00]{Ibort&Magri&Marmo}
A.~Ibort, F.~Magri, and G.~Marmo, \emph{Bihamiltonian structures and {S}täckel
  separability}, Journal of Geometry and Physics \textbf{33} (2000), no.~3--4,
  210--228.

\bibitem[KJ80]{Kalnins&Miller80}
E.~G. Kalnins and Willard~Miller Jr., \emph{{K}illing tensors and variable
  separation for {H}amilton-{J}acobi and {H}elmholtz equations}, SIAM Journal
  on Mathematical Analysis \textbf{11} (1980), no.~6, 1011--1026.

\bibitem[LC04]{Levi-Civita}
Tullio Levi-Civita, \emph{Sulla integrazione della equazione di
  {H}amilton-{J}acobi per separazione di variabili}, Mathematische Annalen
  \textbf{59} (1904), no.~3, 383--397.

\bibitem[LC04]{Lim&Carminati.I}
Allan E.~K. Lim and John Carminati, \emph{The determination of all syzygies for
  the dependent polynomial invariants of the {R}iemann tensor. {I}. {P}ure
  {R}icci and pure {W}eyl invariants}, Journal of Mathematical Physics
  \textbf{45} (2004), no.~4, 1673--1698.

\bibitem[Lun01]{Lundmark}
Hans Lundmark, \emph{{N}ewton systems of cofactor type in {E}uclidean and
  {R}iemannian spaces}, Dissertations, vol. 719, Linköping Studies in Science
  and Technology, Linköping, 2001.

\bibitem[MMS04]{McLenaghan&Milson&Smirnov}
Raymond~G. McLenaghan, Robert Milson, and Roman~G. Smirnov, \emph{{K}illing
  tensors as irreducible representations of the general linear group}, Comptes
  Rendus de l'Académie de Sciences Paris \textbf{339} (2004), 621--624.

\bibitem[MT98]{Matveev&Topalov98}
Vladimir~S. Matveev and Peter~J. Topalov, \emph{Trajectory equivalence and
  corresponding integrals}, Regular and Chaotic Dynamics \textbf{3} (1998),
  no.~2, 30--45.

\bibitem[MT00]{Matveev&Topalov00}
\bysame, \emph{Metric with ergodic geodesic flow is completely determined by
  unparametrised geodesics}, ERA-AMS \textbf{6} (2000), 98--104.

\bibitem[Nij51]{Nijenhuis}
Albert Nijenhuis, \emph{{$X_{n-1}$}-forming sets of eigenvectors}, Proceedings
  of the Koninklijke Nederlandse Akademie van Wetenschappen \textbf{54} (1951),
  200--212.

\bibitem[RW09]{Rauch-Wojciechowski}
Stefan Rauch-Wojciechowski, \emph{From {J}acobi problem of separation of
  variables to theory of quasipotential {N}ewton equations}, Regular and
  Chaotic Dynamics \textbf{14} (2009), no.~4-5, 550--570.

\bibitem[RWML99]{Rauch-Wojciechowski&Marciniak&Lundmark}
Stefan Rauch-Wojciechowski, Krzysztof Marciniak, and Hans Lundmark,
  \emph{Quasi-{L}agrangian systems of {N}ewton equations}, Journal of
  Mathematical Physics \textbf{40} (1999), no.~12, 6366--6398.

\bibitem[Sch]{Schoebel}
Konrad~P. Schöbel, \emph{Integrable {K}illing tensors on the {$3$}-sphere}, in
  preparation.

\bibitem[St{\"a}97]{Staeckel}
Paul St{\"a}ckel, \emph{Über die {I}ntegration der {H}amilton'schen
  {D}ifferentialgleichung mittels {S}eparation der {V}ariablen}, Mathematische
  Annalen \textbf{49} (1897), no.~1, 145--147.

\bibitem[Tho86]{Thompson}
G.~Thompson, \emph{{K}illing tensors in spaces of constant curvature}, Journal
  of Mathemaical Physics \textbf{27} (1986), no.~11, 2693--2699.

\bibitem[Wol98]{Wolf98}
Thomas Wolf, \emph{Structural equations for {K}illing tensors of arbitrary
  rank}, Computer Physics Communications \textbf{115} (1998), no.~2-3,
  316--329.

\end{thebibliography}

\end{document}